\documentclass[a4paper,11pt,DIV=9]{scrartcl}
\usepackage[T1]{fontenc}
\usepackage{multicol}
\usepackage{amsmath}
\usepackage{amssymb}
\addtokomafont{disposition}{\rmfamily\normalfont\scshape}
\usepackage{mathtools}
\usepackage[osf,sc]{mathpazo}
\usepackage{relsize}
\usepackage{amsthm}
\usepackage{enumerate}
\usepackage[autostyle=true]{csquotes}
\usepackage{verbatim}
\usepackage{tikz}
\usepackage{tikz-cd}
\usepackage{microtype}
\usepackage{pdfpages}
\usepackage[british]{babel}
\usepackage[hidelinks,unicode]{hyperref}
\recalctypearea

\usepackage[T1]{fontenc}
\usepackage{microtype}
\usepackage{lscape}
\usepackage{etoolbox}
\usepackage{adjustbox}
\usepackage{graphicx}
\usepackage{booktabs} % for professional tables
\usepackage{comment}
\usepackage{csquotes}
\usepackage{algorithm}
\usepackage{algorithmic}
\usepackage{comment}
\usepackage{thmtools}
\usepackage{thm-restate}
\usepackage{caption}
\usepackage{subcaption}
\usepackage{url}
\usepackage{placeins}
\usepackage{natbib}

\usepackage{amsmath}
\usepackage{amssymb}
\usepackage{mathtools}
\usepackage{amsthm}
\usepackage{siunitx}
\usepackage{cleveref}

\newtheorem{theorem}{Theorem}[section]
\newtheorem*{theorem*}{Theorem}

\theoremstyle{definition}
\newtheorem{definition}[theorem]{Definition}

\theoremstyle{remark}
\newtheorem{remark}[theorem]{Remark}

\usepackage[textsize=tiny]{todonotes}

\newtoggle{arxiv}
\togglefalse{arxiv}
\newcommand{\PerD}{\mathbf{P}}
\newcommand{\expP}{\mathbb{E}\mathbf{P}}
\newcommand{\mshift}{\mathbf{MS}}

\newcommand{\death}{\mathbf{D}}

\newcommand{\NIPH}{\textsmaller{NIPH}}

\newcommand{\TDA}{\textsmaller{TDA}}
\newcommand{\PH}{\textsmaller{PH}}
\newcommand{\VR}{\textsmaller{VR}}

\newcommand{\R}{\mathbb{R}}

\DeclareMathOperator{\diag}{diag}

\newcommand\michael[1]{\noindent{\textcolor{magenta}{[MTS: #1]}}}
\newcommand\vincent[1]{\noindent{\textcolor{magenta}{[VPG: #1]}}}
\renewcommand\vincent[1]{}\renewcommand\michael[1]{}

\title{Non-isotropic Persistent Homology: \\ \smaller{Leveraging the Metric Dependency of \textsmaller{PH}}% Using Optimal Transport
}
\author{Vincent P Grande\footnote{\textsmaller{RWTH} Aachen University, \textsmaller{VPG} and \textsmaller{MTS} acknowledge funding by the German Research Council (\textsmaller{DFG}) within Research Training Group 2236 (UnRAVeL).}~~and Michael T Schaub\footnote{\textsmaller{RWTH} Aachen University, \textsmaller{MTS} acknowledges partial funding by the Ministry of Culture and Science (\textsmaller{MKW}) of the German State of North Rhine-Westphalia ("\textsmaller{NRW} R\"uckkehrprogramm") and the European Union (\textsmaller{ERC}, \textsmaller{HIGH-HOPeS}, 101039827). Views and opinions expressed are however those of the author(s) only and do not necessarily reflect those of the European Union or the European Research Council Executive Agency. Neither the European Union nor the granting authority can be held responsible for them.}}
\date{\vspace{-5ex}}
\begin{document}
	\maketitle
	\begin{abstract}
		Persistent Homology is a widely used topological data analysis tool that creates a concise description of the topological properties of a point cloud based on a specified filtration.
		Most filtrations used for persistent homology depend (implicitly) on a chosen metric, which is typically agnostically chosen as the standard Euclidean metric on $\mathbb{R}^n$.
		Recent work has tried to uncover the \enquote{true} metric on the point cloud using distance-to-measure functions, in order to obtain more meaningful persistent homology results.
		Here we propose an alternative look at this problem:
		we posit that information on the point cloud is lost when restricting persistent homology to a single (correct) distance function.
		Instead, we show how by varying the distance function on the underlying space and analysing the corresponding shifts in the persistence diagrams, we can extract additional topological and geometrical information.
		Finally, we numerically show that non-isotropic persistent homology can extract information on orientation, orientational variance, and scaling of randomly generated point clouds with good accuracy and conduct some experiments on real-world data.
	\end{abstract}
	
	\begin{figure*}[ht!]
		\begin{center}
			\includegraphics[width=\textwidth]{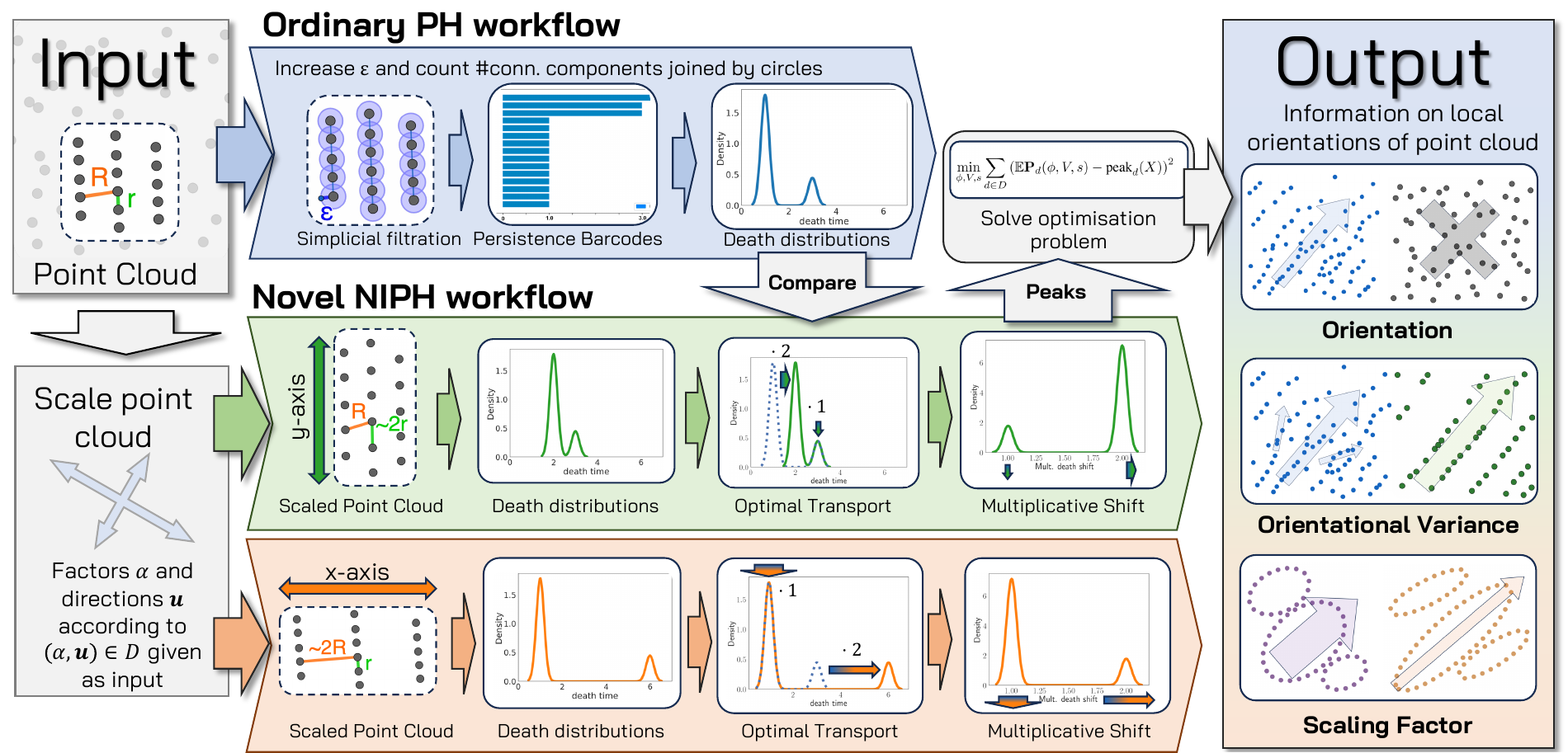}
			\caption{\textbf{Schematic of Non-Isotropic Persistent Homology (\NIPH{}).} \\
				\textbf{Step~1.}
				Produce different versions of the input point cloud by applying directional scaling in direction $\mathbf{u}$ and scaling-factor $s$ according to elements in input set $D$.
				\textbf{Step~2.} Compute \PH{} of desired degree (the diagram displays 0\textsuperscript{th} persistent homology) for input point cloud and each of the scaled point clouds.
				\textbf{ Step~3.} Compute optimal transport between weighted death time distributions of input point cloud and each of the scaled point clouds.
				\textbf{ Step~4.} Compute multiplicative factor of shift for each death time.
				Extract maxima and use optimisation problem to compute preferred orientations, scaling-factor, and orientational variance of point cloud, as seen in the illustrations on the right.
			}
			\label{fig:fig1}
		\end{center}
	\end{figure*}
	\section{Introduction}
	%\vincent{
		%We want to communicate: The multiplicative shift diagram of NIPH is already a really great statistics, and you can extract information from that in multiple ways.
		%See all the comparison to the normal death density diagrams resembling ordinary PH.
		%It is soooo much easier to read!
		%And, on top of that, we have come up with many ways how to exploit it...
		%}
	Over the last decades, topological data analysis (\TDA{}) has proven to provide a valuable toolkit for extracting information out of complex data sets.
	Most notably, persistent homology (\PH{}) provides a straight-forward way to extract topological information across different scales from a point cloud.
	The resulting persistence diagrams and persistence barcodes form a metric space and have been used for many interesting applications, see for example \cite{Edelsbrunner2008}.
	An important motivation for our work is to combine the robust topological descriptors \PH{} provides of (point-cloud) data with more refined geometric notions of orientation and preferred dimensions.
	
	In order to form the simplicial filtration used for computing \PH{}, we need to specify a distance function on the point cloud, which is typically chosen as a metric on $\R^n$; most often the standard $\ell_2$ metric on $\R^n$ is used.
	As in topology virtually all reasonable metrics on $\R^n$ are equivalent and induce the same topology, it would seem that the choice of metric for \PH{} is irrelevant.
	However, this is not true:
	because the metric controls the birth-time of simplices, changing the metric alters the birth and death times of the topological features of the constructed simplicial complexes.
	A change in metric may even eradicate certain topological features or introduce new topological features.
	The role of the underlying metric for \PH{} is often not further investigated, even though it has been acknowledged in recent work such as \cite{chazal2011geometric, Anai2020}, although from a different view.
	However, this influence is treated by the authors as a problem that they attempt to fix by introducing an improved Euclidean distance that makes \PH{} more robust to outliers.
	%There has been some recent work on how to adapt the filtration and the underlying distance function to make persistent homology more robust against outliers \cite{Anai2020}.
	Stated differently, the aim is to construct a single metric that leads to the ``best possible'' simplicial filtration that provides an accurate topological description of the underlying space (point cloud).
	
	\begin{figure}[b!]
		\begin{center}
			\begin{subfigure}{0.46\columnwidth}
				\centerline{\includegraphics[width=\columnwidth]{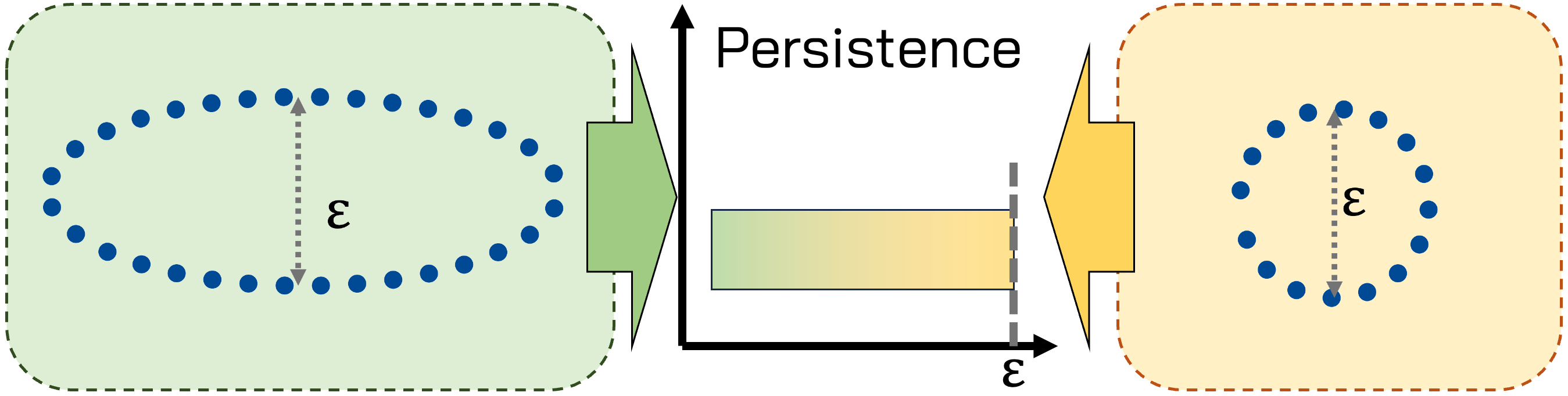}}
			\end{subfigure}
			\begin{subfigure}{0.06\columnwidth}
				\centerline{}
			\end{subfigure}
			\begin{subfigure}{0.46\columnwidth}
				\centerline{\includegraphics[width=\columnwidth]{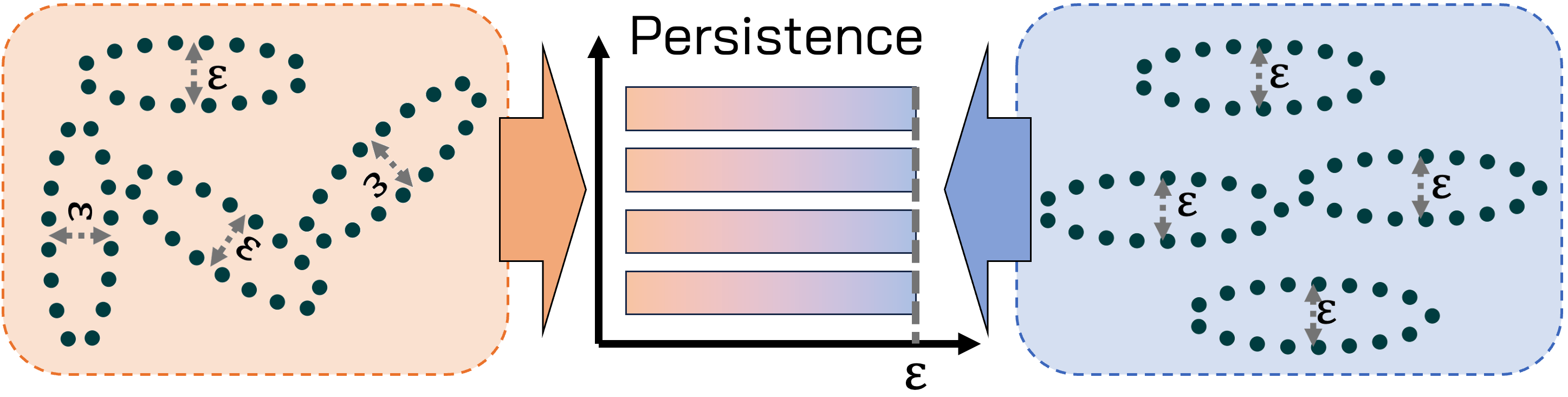}}
			\end{subfigure}
			\caption{\textbf{Data features not captured by standard \PH{}.} \emph{Left:} Persistent homology will not distinguish circles and ellipses by their death time $\varepsilon$. \emph{Right:} Persistent homology cannot distinguish orientations and orientational variances of the data set by death times $\varepsilon$.}
			\label{fig:problem1}
		\end{center}
	\end{figure}
	
	In this work, we take a different perspective on the influence of the choice of a metric on \PH{}.
	Rather than trying to find the best possible metric, we ask: how does \PH{} change as we vary the metric, and can we exploit the induced changes in \PH{} to extract additional information from the data?
	We name this approach non-isotropic persistent homology (\NIPH{})\,---\,see~\Cref{fig:fig1}. 
	Note that \NIPH{} not only provides us with a means to assess how robust (sensitive) the results of \PH{} are to changes in the metric.
	%In contrast, we propose to use the influential role of the underlying metric rather as a tool to generate additional meaningful robust features of the point cloud.
	In addition, non-isometric persistent homology (\NIPH{}) harvests the rich information of how the persistence diagram of a point cloud changes when we change the underlying metric, to extract new information that is not apparent from any single \PH{} analysis (See \Cref{fig:problem1}).
	This may be seen as analogous to cases in physics, where taking the derivative of one important physical quantity often yields another physical quantity of interest.

	As a concrete example, the notion of a preferred orientation is perhaps the most intuitive quantity that is sensitive to, and therefore measurable by a change in the metric\,---\,certain coordinate rescalings will show virtually no effect, whereas others will strongly alter the \PH{} diagram.
	Indeed, orientations and preferred direction are omnipresent in data:
	Animal biologists study the intricate patterns of coordinated swarm behaviour of insects, birds, and marine life \citep{schneider1995}.
	Soft and condensed matter physicist study the patterns and emerging orientations in molecules shifting between solid and fluid state \citep{marder2010}.
	Microbiologists are interested in the patterns and directions viral \textsmaller{RNA} infiltrates cells and macrophages come to the rescue \citep{steenblock2021}.
	Judging based on this prevalence of notions of orientations and direction in applications, making persistent homology sensitive to orientations presents itself to be a fruitful endeavour.
	
	From an abstract point of view, a set of points $X$ only carries a geometric meaning in conjunction with a metric $d$.
	The metric encodes information on the distance between individual pairs of points.
	Taken together, this then determines the overall geometry of the space.
	Importantly, the metric determines the topology of the underlying space as well.
	From a different perspective, this means that we can \emph{change} the geometry of a space by \emph{changing} the metric of the space.
	Making use of this simple, but powerful insight is the key idea behind \NIPH.
	
	%We briefly give an overview over some of the features that can be extracted by \NIPH{}, see \Cref{fig:fig1} \emph{right}:
	%			\textbf{Orientation.}
	%			Two point clouds with similar $0$-dimensional \PH{}.
	%			\emph{Left:} The point cloud has additional structure: a preferred orientation.
	%			\emph{Right:} There is no preferred orientation in the point cloud.
	%			\NIPH{} can detect the differences in the amount of orientation in both point clouds across arbitrary scales.
	%			
	%			\textbf{Orientational variance.} Two point clouds with similarly structured $0$-dimensional \PH. Both have additional structure and a preferred dimension.
	%			\emph{Left:} There is a comparably large variance in the local preferred direction.
	%			\emph{Right:} The alignment of the point data with the preferred direction is very strong.
	%			\NIPH{} can detect the differences in the amount of variance of orientation in both point clouds.
	%			
	%			\textbf{Scaling factor.}
	%			Two point clouds with similarly structured $1$-dimensional \PH.
	%			Both have additional structure and are scaled along the same preferred dimension.
	%			The ellipses of the \emph{left} point cloud are scaled differently than the ellipses on the \emph{right}.
	
	\paragraph{Organisation of the paper} 
	In \Cref{sec:Background}, we will give an overview of some of the core concepts of \NIPH{}, including Simplicial Complexes, the Vietoris--Rips filtration, Persistent Homology, and Optimal Transport.
	In \Cref{sec:methods}, we give an overview over the \NIPH{} algorithm.
	We provide an introductory example in \Cref{subsec:IntroductoryExample}, and the general method in \Cref{subsec:GeneralMethod}, highlight further approaches using the same framework in \Cref{subsec:AdditionalApproaches}, and finally give theoretical guarantees in \Cref{subsec:TheoreticalGuarantees}.
	Finally, we validate the performance of \NIPH{} in experiments on synthetic and real-world data in \Cref{sec:Experiments}.
	The code to replicate all experiments in this paper can be found here \url{https://git.rwth-aachen.de/netsci/publication-2023-non-isotropic-persistent-homology}.
	
	\paragraph{Related Work}
	\cite{chazal2011geometric} introduced distance-to-measure filtrations to adapt the distance function (or the simplicial filtration) of the point cloud to reduce the effect of outliers (cf. \cite{Anai2020}).
	There has also been work on extracting geometric information using persistent homology, such as the persistent homology transform \cite{Turner2014}.
	However, this work focussed on extracting shapes of $2d$ surfaces and $3d$ objects in $3d$ space, whereas we focus on local geometric information encoded in point clouds.
	In \cite{Hofer2019} and \cite{Carriere2021}, the authors introduce the notion of differentiating persistent homology diagrams.
	However, their goal is to utilize this differentiability to make \PH{} accessible to machine learning tasks, rather than to extract geometrical information on the point cloud.
	In \cite{Stucki2023}, the authors consider the problem of matching persistence diagrams as well, but for persistence diagrams from the same metric space.
	
	\section{Theoretical Background}
	\label{sec:Background}
	Our method of Non-Isotropic Persistent Homology draws from many ideas of Algebraic Topology, Topological Data Analysis, Computational Geometry, and Transport Theory.
	We give a brief overview over the main concepts needed for \NIPH{}, but direct the interested reader to the more comprehensive introductions to Algebraic Topology \citep{Bredon:1993,Hatcher:2002}, Topological Data Analysis \citep{Chazal2021}, Computational Geometry \citep{Preparata2012}, and Optimal Transport \citep{Peyre2019}.
	\paragraph{Simplicial Complexes}
	%\vincent{Do I want to have a figure for that?}\michael{I think the paper will be unreadable for somebody not familiar with a bit of topology background. A figure of an SC does not help for that, I think}
	Simplicial Complexes are higher-order generalisations of graphs, originally constructed in algebraic topology to capture the shapes (homotopy types) of topological spaces.
	\begin{definition}[Simplicial Complex]
		A simplicial complex $X$ consists of a set of vertices $V$ and a set of finite non-empty subsets of $V$, called the simplices $S$, such that \textbf{(i)} $S$ is closed under taking non-empty subsets and \textbf{(ii)} for every $v\in V$ the singleton set $\{v\}$ is contained in $S$. We call the simplices with $k+1$ elements the $k$-simplices.
	\end{definition}
	A graph can be considered as a simplicial complex with the nodes being the $0$-simplices and the edges being the $1$-simplices, connecting two nodes.
	Intuitively, $2$-simplices are represented by triangles between three nodes, $3$-simplices are tetrahedra between four nodes, and so on.
	\paragraph{Vietoris--Rips filtration}
	The Vietoris--Rips (\VR{}) filtration is a tool from computational topology that turns a point cloud into a filtration of simplicial complexes.
	\begin{definition}[Vietoris--Rips Complex]
		Given a point cloud $X$ with distance function $d$ and a parameter $r\ge 0$, the associated simplicial complex $\text{VR}_r(X)=(X,S)$ is given by the vertex set $X$ and the set of simplices
		\[
		S=\left\{\sigma\subset X : \max_{x,y\in \sigma} d(x,y)\le r\right\}.
		\]
	\end{definition}
	For $r\le r'$, we then have the canonical inclusion $\text{VR}_r(X)\subset\text{VR}_{r'}(X)$.
	Intuitively, we can construct $\text{VR}_r(X)$ by connecting all points with a distance of at most $r$ with edges, and finally adding all available higher-order simplices (i.e. the ones where all edges between vertices are already included).
	\paragraph{Persistent Homology}
	Homology is a tool from algebraic topology that describes the \enquote{shape} of a topological space.
	From a perspective of computational topology, homology groups have two major advantages over the homotopy groups, another popular concept of algebraic topology.
	Firstly, the linear algebraic definition of homology groups makes them easily accessible for computational methods \citep{Bauer2021}.
	Secondly, the theory of homology groups is well understood, whereas the computation of higher homotopy groups of even the most basic non-trivial topological spaces, spheres, is an active research topic today (See for example \cite{Ravenel2023}).%\vincent{We might drop this sentence, although I haven't read this perspective in other TDA papers and I like it.}
	%\michael{why drop it? OK to keep it in my opinion}
	
	Intuitively, $n$-dimensional homology encodes $n$-dimensional holes in the topological space.
	$0$-dimensional holes are connected components, $1$-dimensional holes are loops, $2$-dimensional holes are cavities in $3d$-space and so on.
	
	\begin{adjustbox}{valign=b, minipage=0.65 \linewidth}
		Given a point cloud $X$, persistent homology keeps track of these homological features across all stages of the associated \VR{} filtrations.
		Persistent homology summarises these findings in a persistence diagram that, for each connected component/loop/cavity/etc./ $F$, stores $r$ and $r'$ such that the generalised loop $F$ first occurred in the step $\text{VR}_r(X)$ of the \VR{} filtration and last occurred in $\text{VR}_{r'}(X)$.
		(Note: for $0$-dimensional features, the birth time $r$ is always $0$.)
		For loops for instance, the birth time marks the time the loop is closed, whereas the death time $r'$ marks the filtration step where the loop is filled by higher-order simplices.
		We use the implementation by \cite{gudhi:urm}.
	\end{adjustbox}
	\hfill%
	\begin{adjustbox}{valign=b, minipage=0.28\linewidth}
		\includegraphics[width=\linewidth]{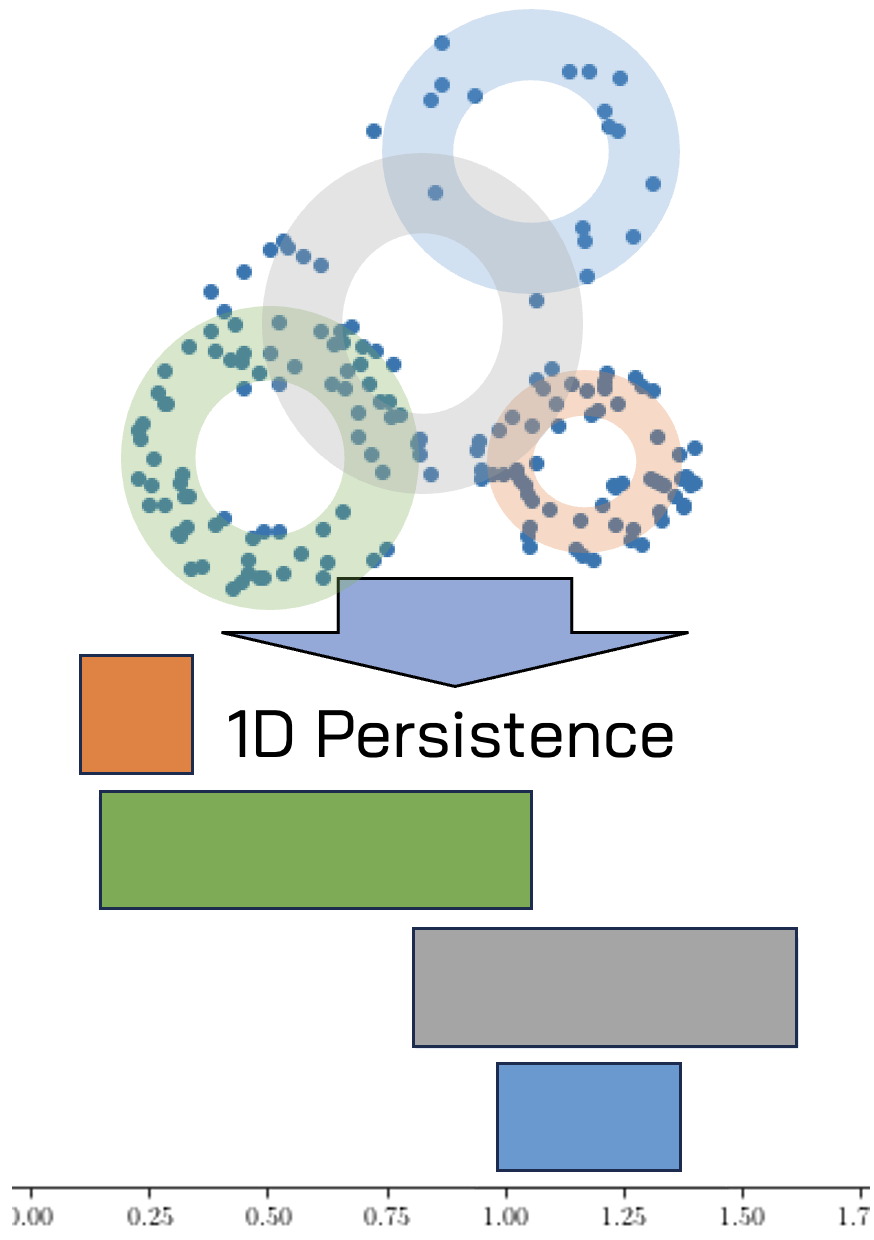}
	\end{adjustbox}
	
	%\vincent{I really don't like this section.}
	%\vincent{Do I want a formal definition here? A picture of a persistence diagram at least?}
	%\vincent{Do I put the many applications here?}
	%\michael{maybe a small picture illustrating a point cloud and associated PH diagram? perhaps as an "inline" figure to save a bit of space?! (see \url{https://www.graphics.rwth-aachen.de/media/papers/345/surface-maps-via-adaptive-triangulations.pdf} for what I mean with this -- also have a look at fig. 15)}
	
	%\paragraph{Metric Spaces}
	
	%
	%\begin{definition}[Metric space]
	%    A metric space $(X,d)$ consists of a set of points $X$ and a symmetric function (the \emph{metric}) $d\colon X\times X\to \R_{\ge 0}$ under the canonical $S_2$-action, such that for all points $x,y,z\in X$: 
	%	\textbf{(i)} $d(x,y)=0 \Leftrightarrow x=0$ and \textbf{(ii)} $d(x,z)\le d(x,y)+d(y,z)$. \vincent{That's concise, but is it good?}
	%    \michael{mmmh, say what this means? triangle inequality, etc.?}
	%\end{definition}
	%\michael{actually, the definition of a metric is not really of interest to us (and if somebody does not know this, we are lost anyway), it is not the properties of the metric we are using but the remark about metric spaces above. I think that aspect could be moved to the intro, as indicated above, and then we remove the metric definition here?!}
	
	\paragraph{Optimal Transport}
	At its core, optimal transport is an optimisation problem.
	Intuitively, the most basic (``earth movers'') formulation asks for the \enquote{cheapest} way to move one arrangement of mass to another, desired arrangement.
	More formally we have the following:
	%More formally, given two distributions on measurable spaces and a cost function between them, optimal transport looks for the minimal cost and the associated transport plan to transform one distribution into the other.
	
	\begin{definition}[Optimal transport]
		Given two measurable spaces $A$ and $B$, a distribution $\mu$ on $A$, a distribution $\nu$ on $B$, and a joint cost function $c\colon A\times B\rightarrow \R_{\ge 0}$, optimal transport looks for a joint distribution $\pi$ on $A\times B$ with marginal distributions $\mu$ and $\nu$ that minimises
		\[
		C(\pi)=\int_{A\times B} c(a,b)d\pi+r(\pi)
		\]
		over the space of all such distributions where $r$ is some regularisation function (often an entropic regularisation). $\pi$ is then called the \emph{transport plan} and $C(\pi)$ the \emph{transport distance}.
	\end{definition}
	Optimal transport has seen widespread application in the wider \textsmaller{ML} community: \cite{rubner2000} in Computer Vision, \cite{kolouri2020} in graph learning, \cite{kandasamy2018} for neural architectures, \cite{gramfort2015} in neuroimaging, etc.
	In \TDA{}, optimal transport has long been used to compare homological information across spaces.
	This has been done using the $\ell_\infty$ bottleneck distance \cite{kerber2017} on the level of persistence diagrams, or more directly on the level of persistence landscapes \cite{Bubenik2015}.
	
	\section{Methods}
	\begin{figure}[tb!]
		\begin{center}
			\begin{subfigure}{0.32\columnwidth}
				\includegraphics[width=\columnwidth]{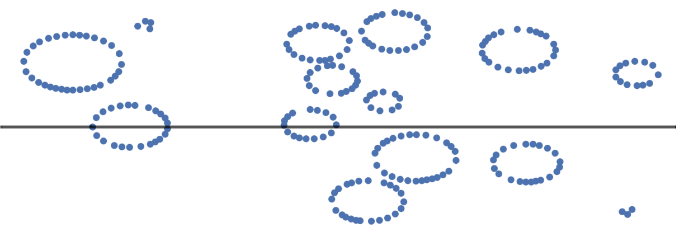}
				\subcaption{$0^\circ$}
			\end{subfigure}
			\begin{subfigure}{0.32\columnwidth}
				\includegraphics[width=\columnwidth]{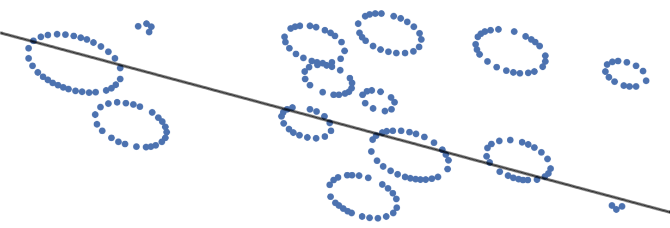}
				\subcaption{$15^\circ$}
			\end{subfigure}
			\begin{subfigure}{0.32\columnwidth}
				\includegraphics[width=\columnwidth]{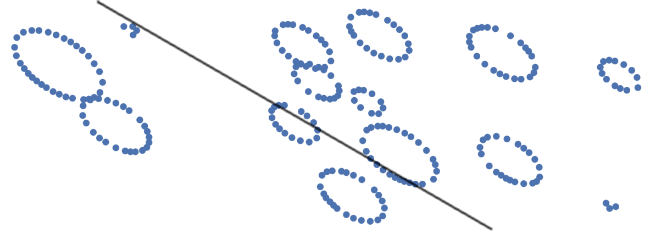}
				\subcaption{$30^\circ$}
			\end{subfigure}
			\begin{subfigure}{0.32\columnwidth}
				\includegraphics[width=\columnwidth]{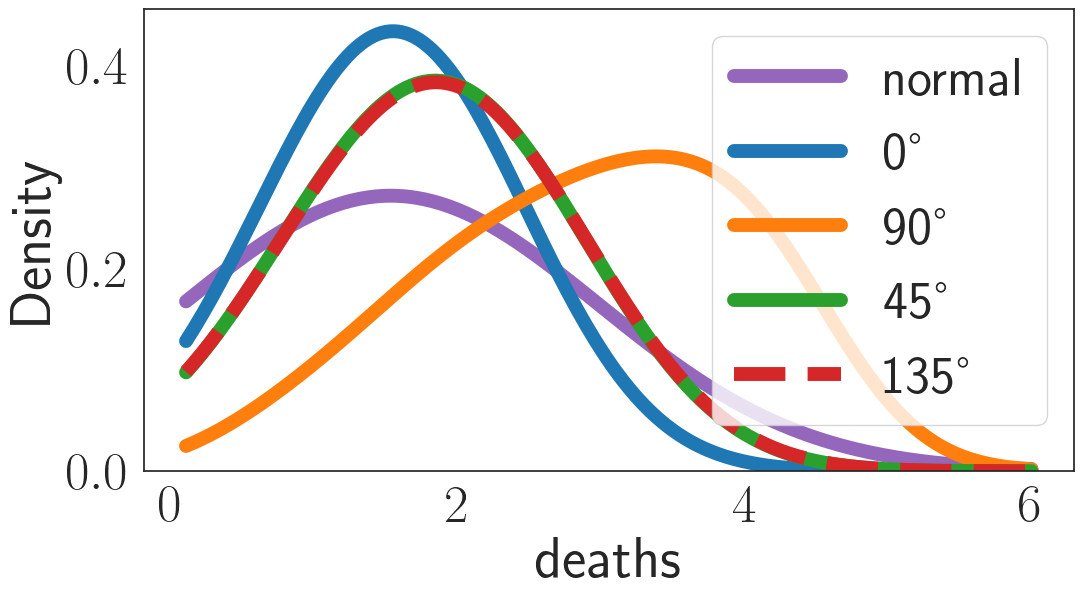}
			\end{subfigure}
			\begin{subfigure}{0.32\columnwidth}
				\includegraphics[width=\columnwidth]{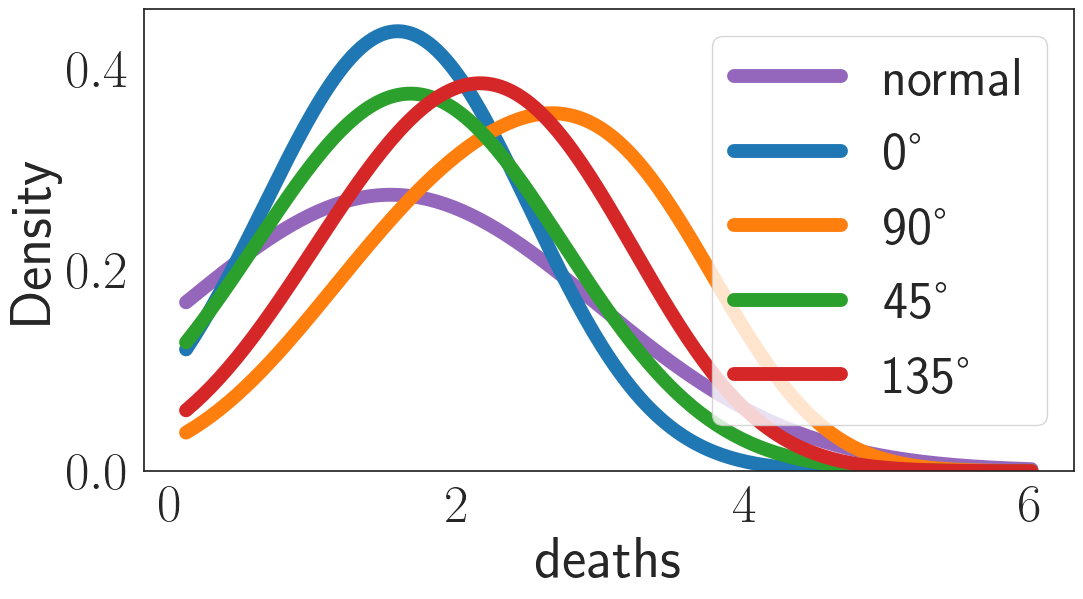}
			\end{subfigure}
			\begin{subfigure}{0.32\columnwidth}
				\includegraphics[width=\columnwidth]{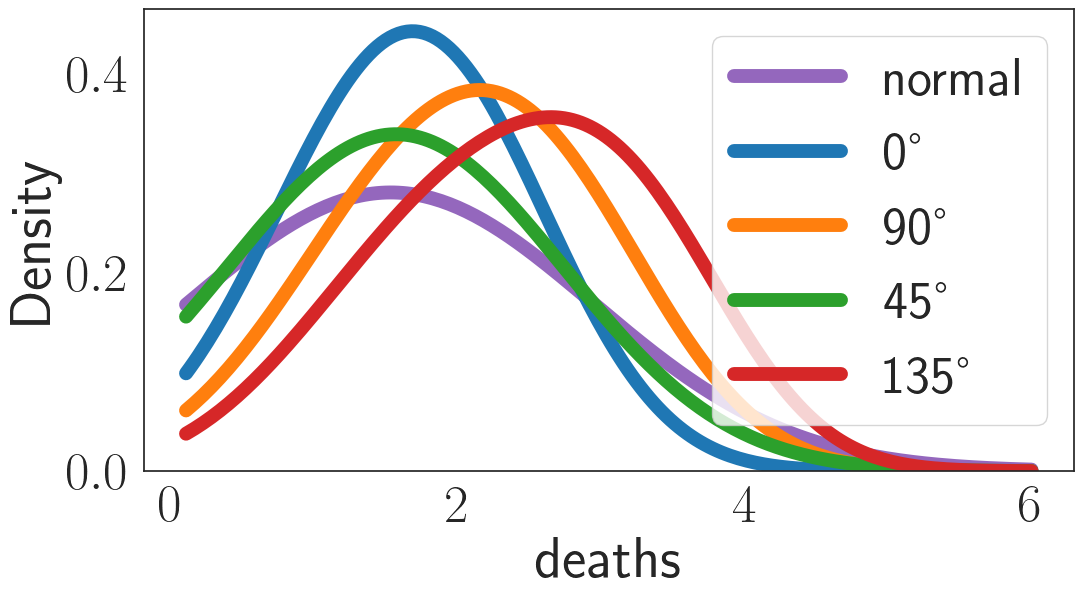}
			\end{subfigure} 
			\begin{subfigure}{0.32\columnwidth}
				\includegraphics[width=\columnwidth]{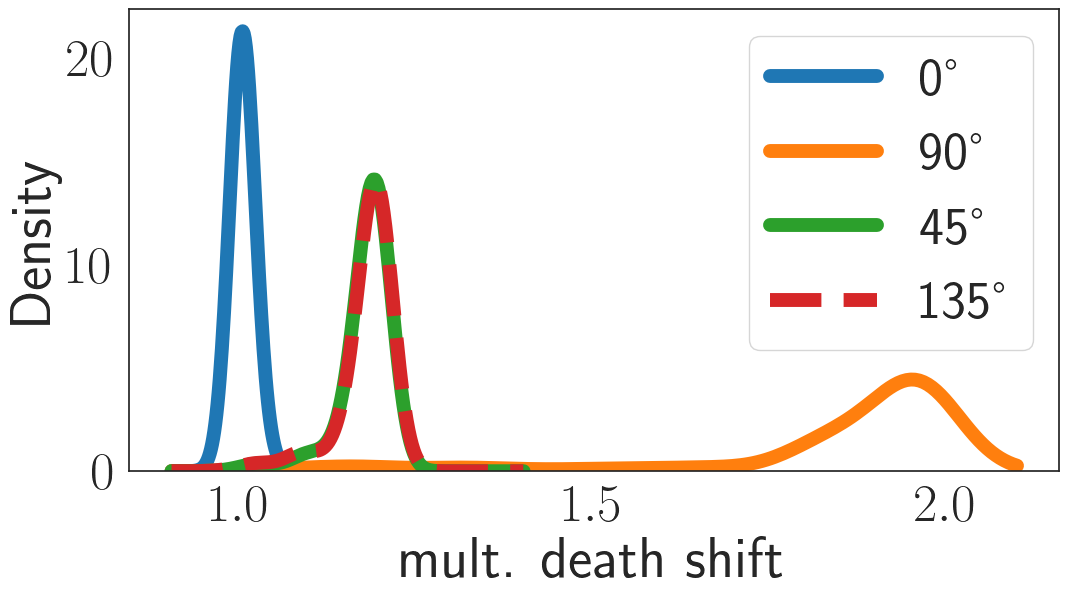}
			\end{subfigure}
			\begin{subfigure}{0.32\columnwidth}
				\includegraphics[width=\columnwidth]{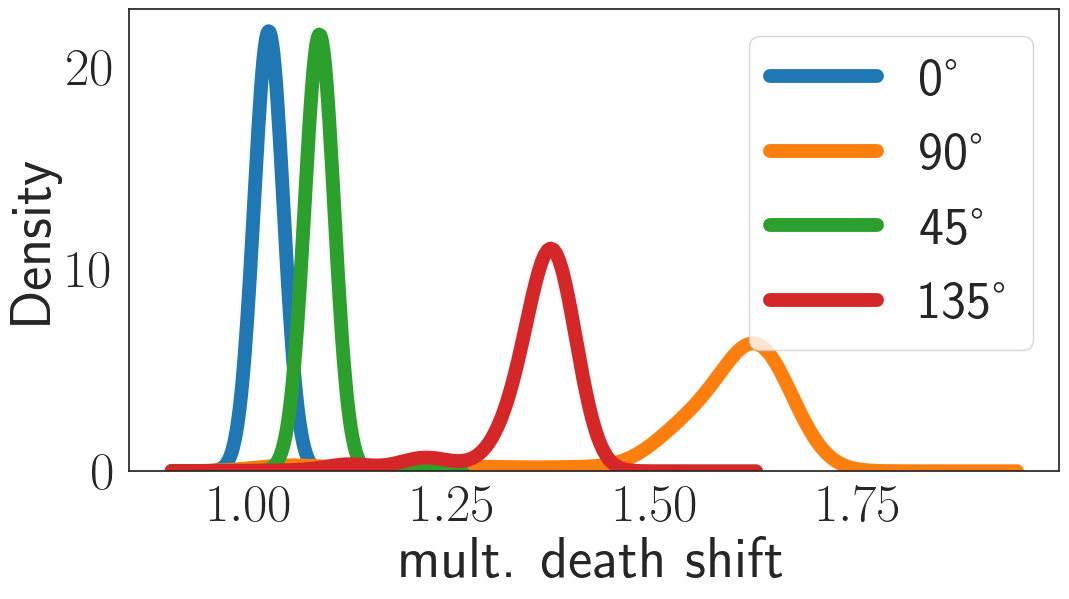}
			\end{subfigure}
			\begin{subfigure}{0.32\columnwidth}
				\includegraphics[width=\columnwidth]{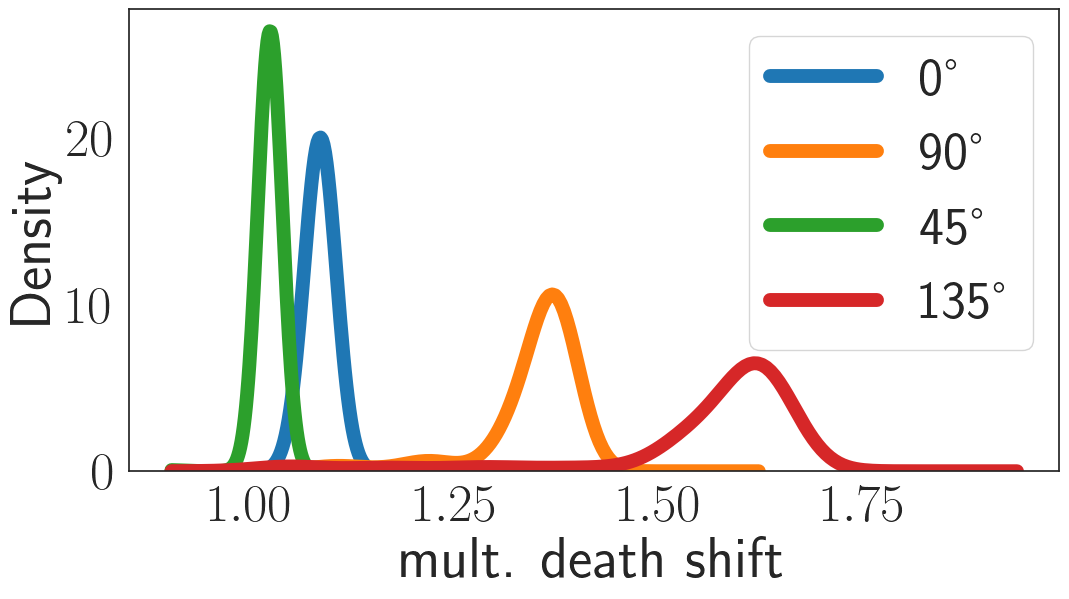}
			\end{subfigure}
			\caption{\textbf{Illustrative Experiments.} Point cloud with ellipses of various radii in orientation $\varphi = 0^\circ$, $15^\circ$ and $30^\circ$, scaling factor $s=2$ and orientational variance $V=0$. \emph{Top:} Part of point cloud. \emph{Middle:} Death density diagram of $1$-dimensional persistent homology. \emph{Bottom:} Multiplicative shift diagram used to extract information on orientation, orientational variance, and scaling.
				In the left diagram, the blue curve represents scaling parallel to the directional scaling of the ellipses.
				Thus there is no change in the death times and the peak of the curve is at $\sim 1.0$.
				The orange curve represents scaling in a direction almost orthogonal to the scaling of the data points, hence the death times are multiplied by the factor of the scaling.
				This is represented by the peak almost reaching $2.0$.
				The red and green curve represent scaling roughly at a direction of $45^\circ$ to the original scaling in the point cloud.
				When we change the orientation of the point cloud the orange peak will move to the left and the the red peak will move to the right.
				The mult.\ death shift diagrams are easy to interpret and concise.
			}
			\label{fig:deathshifts}
			\label{fig:SynthExpSketch}
		\end{center}
	\end{figure}
	\label{sec:methods}
	\subsection{Introductory Example}
	\label{subsec:IntroductoryExample}
	As an introductory example, we can consider the following class of metrics on $\R^2$:
	\begin{definition}
		\label{def:introduction}
		For real $\alpha,\beta > 0$, we define the associated metric $d_{\alpha,\beta}$ on $\R^2$ as follows:
		\[
		\smash{d_{\alpha,\beta}\colon \left( (x_1,y_1)^\top,(x_2,y_2)^\top\right)\mapsto \sqrt{\alpha\left(x_1-x_2\right)^2+\beta \left( y_1-y_2\right)^2}},
		\]
	\end{definition}
	which is the standard Euclidean distance for $\alpha = \beta =1$.
	Picking $\alpha$ and $\beta$ amounts to (de-)prioritising the $x$- and $y$-axis for our metric.
	When choosing $\alpha\gg\beta$, the distance between two points is almost entirely determined by the distance of their $x$-values.
	When choosing $\beta\gg\alpha$, the distance between two points is almost entirely determined by the distance of their $y$-values, and differences in the $x$-coordinate have almost no influence.
	%On the other hand for $\beta\gg \alpha$, even small differences in the $y$-coordinate have more influence than moderately large differences in the $x$-coordinate.
	
	To build some intuition, assume that we compute ordinary $1$-dimensional \PH{} with $\alpha=\beta=1$ first and compare it  with \PH{} associated to $d_{\alpha,\beta}$ for $\alpha =0.5$ and $\beta=1$.
	On a data set consisting of points sampled from circles with radius $\varepsilon/2$, the death times would remain almost constant:
	although points are now farther apart in the $x$-axis, simplices with edge length of $\varepsilon$ are enough to cover all circles in the $y$-direction.
	On the other hand, if the data set consisted of ellipses stretched by a factor of $2$ along the $y$-axis the persistent homology would change with the change in metric.
	Because the change of the metric is orthogonal to the orientation of the ellipses, their death times will double (\Cref{fig:problem1}).
	
	This simple example motivates why looking at the change of \PH{} under different metrics can be a powerful tool to enrich the standard tools of persistent homology with notions of orientation and preferred directions.
	We give a more rigorous account of our mathematical models below.
	
	\subsection{General Method}
	\label{subsec:GeneralMethod}
	%\michael{changed order of sentences and say what it does... then we can perhaps keep the sentnece} 
	\NIPH{} takes as input a point cloud $X\in\R^n$ and a set of scaling directions (see~\Cref{alg:NIMPH}), and computes a mult.\ shift diagram, from which a variety of interesting properties can be computed.
	For instance, \NIPH{} can extract from this the presence or absence of a \textbf{preferred orientation}, \textbf{orientational variance}, and \textbf{scaling factor} (as shown on the \emph{right}):
	
	\begin{adjustbox}{valign=b, minipage=0.65 \linewidth}
		\textbf{Orientation.}
		Two point clouds with similar $0$-dimensional \PH{}.
		\emph{Left:} The point cloud has additional structure: a preferred orientation.
		\emph{Right:} There is no preferred orientation in the point cloud.
		\NIPH{} can detect the differences in the amount of orientation in both point clouds across arbitrary scales.
		
		\textbf{Orientational variance.} Two point clouds with similarly structured $0$-dimensional \PH. Both have additional structure and a preferred dimension.
		\emph{Left:} There is a comparably large variance in the local preferred direction.
		\emph{Right:} The alignment of the point data with the preferred direction is very strong.
		\NIPH{} can detect the differences in the amount of variance of orientation in both point clouds.
		
	\end{adjustbox}
	\hfill%
	\begin{adjustbox}{valign=b, minipage=0.25\linewidth}
		\includegraphics[width=\linewidth]{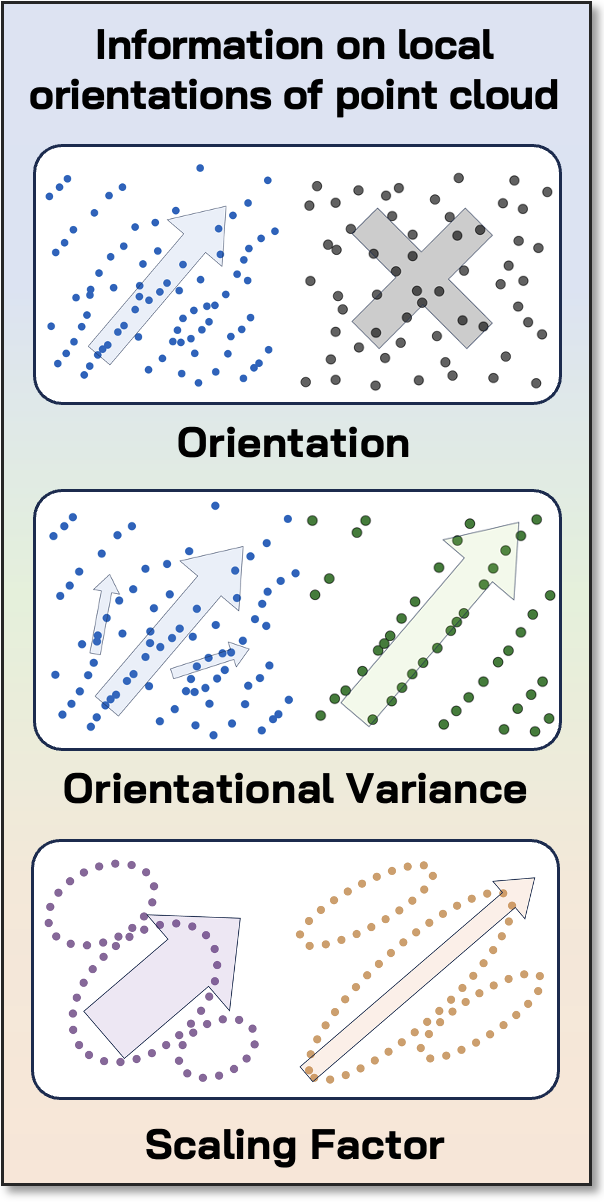}
	\end{adjustbox}
	\textbf{Scaling factor.}
	Two point clouds with similarly structured $1$-dimensional \PH.
	Both have additional structure and are scaled along the same direction.
	The ellipses of the \emph{left} point cloud are scaled differently than on the \emph{right}.
	\paragraph{The distance function}
	For every direction of an $n$-dimensional space given by a unit vector $\mathbf{u}$ we can pick an orthonormal basis with first basis vector $\mathbf{u}$. Let $U$ be the associated base transformation matrix. For a scaling-factor $\alpha>0$ we can define the scaling function
	$
	S_{\mathbf{u},\alpha}\colon x\mapsto U^{-1}\diag((\alpha,1,\dots,1))Ux.
	$
	\begin{definition}
		We compute the distance function $d_{\mathbf{u},\alpha}$ associated to the pair $(\mathbf{u},\alpha)$ as follows:
		\[
		\smash{d_{\mathbf{u},\alpha}\colon (x_1,x_2)\mapsto \|S_{\mathbf{u},\alpha}(x_1)-S_{\mathbf{u},\alpha}(x_2)\|_2.}
		\]
	\end{definition}
	We note that this definition generalises \cref{def:introduction} for $\beta =1$ given in the introduction and does not depend on the concrete choice of $U$.
	Intuitively, $d_{\mathbf{u},\alpha}$ scales the distances in the direction of $\mathbf{u}$ by a factor of $\alpha$ while leaving the orthogonal directions untouched.
	Setting $\alpha = 0$ would amount to the distance obtained by projecting on the subspace orthogonal to $\mathbf{u}$, which would possibly violate the $d(x,y)=0\Leftrightarrow x=y$ condition on a metric.
	This is why we have excluded $\alpha =0$.
	\begin{definition}
		Let $X$ be a point cloud in $n$-dimensional space $X\subset \R^n$, $\mathbf{u}\in \R^n$ a unit vector and $\alpha>0$ a real scaling-factor. We denote by $X_{\mathbf{u},\alpha}$ the scaled data-set
		$
		\smash{X_{\mathbf{u},\alpha}:= \{S_{\mathbf{u},\alpha}(x)\mid x\in X\}.}
		$
	\end{definition}
	It is easy to see that $(X_{\mathbf{u},\alpha},d)$ is isometric to $(X,d_{\mathbf{u},\alpha})$ where $d$ is the Euclidean distance on $\R^n$.
	With this observation in mind, we refer to the metric spaces $(X,d_{\mathbf{u},\alpha})$ as the scaled point clouds.
	\paragraph{Computing persistence diagrams}
	We compute the persistence diagrams $\PerD$ and $\PerD_{\mathbf{u},\alpha}$ on the point clouds $X$ and $X_{\mathbf{u},\alpha}:=(X,d_{\mathbf{u},\alpha})$ for all $(\mathbf{u},\alpha)\in D$ using a Vietoris-Rips filtration on $X$ and $X_{\mathbf{u},\alpha}$.
	\paragraph{Death distributions}
	In many cases, the death of homology classes carries more structural information than their birth time.
	This is because birth time is mainly controlled by point densities, whereas death times measure distances between clusters or the sizes of loops.
	%\michael{do we have  reference or something to back this up a bit, it is believable, but some people may also say: why more information?! can this be proven ?!}
	%\michael{.... if not, let's leave it as is}
	Hence we transform the persistence diagrams to density plots over the death time.
	In the case of $1$-dimensional homology, we can weigh the points according to the difference or quotient of death and birth time, i.e., longer lived loops carry more weight.
	We denote by $\death$ and $\death^{(\mathbf{u},\alpha)}$ the (weighted) vectors of death times.
	E.g., for a grid with grid length $a$ the $0$-dim death distribution $\death$ would just contain entries of $a$.
	\paragraph{Computing the multiplicative shift}
	
	We are interested in how persistent homology and the death distributions change when changing the metric of the underlying data set.
	Hence we propose to compute an optimal transport based matching between the death distribution diagrams of $X$ and of the $X_{\mathbf{u},\alpha}$.
	We can then compute for every point in the original persistence diagram the factor by which the death time was multiplied in the optimal transport matching.
	We can then present these multiplicative death shift factors as another density plot (See \Cref{fig:deathshifts}):
	\begin{definition}[Multiplicative Death Shift]
		Let $T^{(\mathbf{u},\alpha)}$ denote the solution transport matrix of optimal transport between $\death$ and $\death^{(\mathbf{u},\alpha)}$.
		We denote by the $i$-th multiplicative death shift $\mathbf{ms}_i^{\mathbf{u},\alpha}$ associated to scaling ${\mathbf{u},\alpha}$ for $1\le i \le |\PerD|$ the following expression
		\[
		\mathbf{ms}_i^{\mathbf{u},\alpha}(X) := \exp \left(\sum_j T^{\mathbf{u},\alpha}_{i,j}(X)\ln \death^{\mathbf{u},\alpha}_j(X)/\death_i(X) \right).
		\]
		In the most simple case of an $1:1$ transport plan $T$, where $t^{\mathbf{u},\alpha}(i)$ denotes the destination of $i$, this reduces to the much simpler form of
		\[
		\mathbf{ms}_i^{\mathbf{u},\alpha}(X) =\death^{\mathbf{u},\alpha}_{t^{\mathbf{u},\alpha}(i)}/\death_i(X).
		\]
		The multiplicative shift is a scale-free measure of how much the persistence of a homology class is affected by the scaling of the point cloud.
		Now let $w$ denote a vector of weights for the homological features of $X$ in \PH{}.
		Then we denote by $\mshift_{\mathbf{u},\alpha}$ the multiplicative death shift diagram, given by the density diagram
		\[
		\mshift_{\mathbf{u},\alpha}(X) := \operatorname{density}\left\{(\mathbf{ms}_i^{\mathbf{u},\alpha}(X),w_i)\mid 1\le i \le |\PerD|\right\}.
		\]
		$\mshift_{\mathbf{u},\alpha}(X)$ is basically the combination of all the individual $\mathbf{ms}_i^{\mathbf{u},\alpha}(X)$.
		We furthermore denote the $x$-value of the maximum of $\mshift_{\mathbf{u},\alpha}(X)$ by $\text{peak}_{\mathbf{u},\alpha}(X)$, the most prominent mult.\ shift value.
	\end{definition}
	\paragraph{Extracting the orientations}
	The multiplicative death shift density diagrams obtained in the previous step are an interesting object in their own right.
	To illustrate the utility of these diagrams, we give an example of extracting information on orientations ($\varphi$), orientational variance $V$, and scaling factor $s$ from these shift diagrams.
	The idea behind this is that we can compute how the multiplicative shift diagrams should behave under given parameters $(\varphi,V,\alpha)$.
	We do this using the model of rectangles (closely matching the behaviour of ellipses, but having a concise analytical formulation) for $1$-dimensional homology and grids for $0$-dimensional homology.
	\begin{definition}[Expected Peak]
		Let $\expP_{\mathbf{u},\alpha}(\varphi, V, s)$ be a function such that for a point cloud $Y$ sampled from rectangles (a grid) with scaling factor $s$, orientation $\varphi$ and orientational variance $V$ we have that $\mshift_{\mathbf{u},\alpha}(Y)$ in $1$-dimensional ($0$-dimensional) persistent homology takes it maximum value at $\expP_{\mathbf{u},\alpha}(\varphi, V, s)$.
	\end{definition}
	$\expP$ has an analytic description depending on the considered homology dimension, but can be determined via sampling as well (See \Cref{sec:AnalyticDescription}).
	For $D$ being the set of sampling directions and scalings, \NIPH{} solves the optimisation problem
	\[
	\min_{\varphi,V,s} \sum_{d\in D} \left(\expP_d(\varphi, V, s)-\text{peak}_{d}(X)\right)^2
	\]
	to obtain an estimated orientation $\varphi$, orientational variance $V$, and scaling factor $\alpha$ of the point cloud $X$ explaining the witnessed multiplicative shift diagrams best.
	We provide an in-depth discussion of choices of sampling directions and sampling scaling factors $\mathbf{u}$ in practice in \Cref{sec:SelectionProbingDirections}.
	\begin{algorithm}[tb]
		\caption{Non-isotropic persistent homology (\NIPH{})}
		\label{alg:NIMPH}
		\begin{algorithmic}
			\STATE {\bfseries Input:} Point cloud $X$, list of directions and scale-factors $D$
			\STATE Compute persistent homology diagram $\PerD$ of $X$
			\FOR{$d\in D$}
			\STATE Compute scaled data set $X_d$ according to $d=(\mathbf{u},\alpha)$.
			\STATE Compute persistent homology $\PerD_d$ of $X_d$.
			\STATE Solve optimal transport from $\PerD$ to $\PerD_d$.
			\STATE Compute mult. shift diagram $\mshift_d$ associated to $d$
			\ENDFOR
			\STATE Post-Processing, e.g.\ solve optimisation problem for best orientation $\varphi$, scaling-factor $s$, and orientational variance $V$ matching the maxima of $\mshift_d$.
			\STATE \textbf{Output:} $\varphi$, $s$, $V$
		\end{algorithmic}
	\end{algorithm}
	
	\subsection{Additional approaches}
	\label{subsec:AdditionalApproaches}
	The underlying idea of \NIPH{}\,---\,to leverage the metric dependency of persistent homology\,---\, allows for great flexibility and extensions beyond orientations, orientational variance and scaling as discussed above.
	We highlight two additional approaches to extract information from point clouds using the \NIPH{} framework.
	\paragraph{Exotic metrics and outlier detection}
	The family of metrics we considered so far was based on re-weighing the influence of different directions of our ambient space $\R^d$.
	However, we can construct more intricate families of metrics for specific tasks.
	\begin{definition}[Outlier metric, cf. \cite{Anai2020}]
		Let $X\subset \R^d$ be a finite point cloud and $\mu\colon X\rightarrow \R_{>0}$ a function. We then define a metric $d_\mu$ on $X$ by
		$d_\mu(x,y)=2d(x,y)/(\mu(x)+\mu(y)).
		$
	\end{definition}
	In the previous definition, for a $\delta>0 $, we can define $f$ to be
	\[
	f_\delta(x)=\sum_{y\not = x, y\in X}\exp (-d(x,y)^2/\delta)
	\]
	and then set $\mu_\delta(x)=f_\delta(x)\cdot|X|/\sum_{y\in X} f_\delta(y)$.
	For $\delta\rightarrow 0$, this returns the original metric, whereas increasing $\delta$ will increasingly make outliers further apart and decrease the distance in dense parts of the point cloud.
	Adopting this metric and tracking the persistence and multiplicative shift diagrams gives insight into the outlier structure of the base point cloud $X$.
	
	\paragraph{Tracking the orientation of death simplices}
	Persistent Homology equips us with yet another tool:
	For every persistence class, we also get the death simplex, i.e.\ the simplex that connects the connected components, closes the loop, etc.
	We can then look at the orientation $\varphi_\sigma$ of the longest edge of the death simplex $\sigma$. % (which is the entire simplex for $0$-homology).
	By considering the entire distribution of orientations, we can again infer information on the geometry of the point cloud.
	However, limiting this analysis to the snapshot of the base point cloud with standard metric is prone to noise:
	Even a circle in one-dimensional homology, despite having no preferred orientation, will have a death edge with an orientation.
	Hence we propose to apply the \NIPH{} framework to the analysis of death simplices:
	By tracking the distributions of death edge orientations over different metrics on the point cloud, we can distinguish noise from real orientations, and obtain the underlying scaling factor of the orientations for free.
	\subsection{Theoretical Guarantees}
	\label{subsec:TheoreticalGuarantees}
	%In the previous sections, we have introduced the underlying concepts of \NIPH{} and explained the steps of the algorithm. Furthermore, we have given proof of the flexibility of the framework.
	\begin{adjustbox}{valign=b, minipage=0.73 \linewidth}
		\begin{restatable}[Theoretical guarantees]{theorem}{TheoremGuarantees}
			\label{thm:cleangrid}
			Given points on an orthogonal $n_1\times n_2$ grid with $n_1,n_2>1$ in $\R^2$ rotated by $\varphi$ with distances $d_1$ and $d_2$ with $d_1<d_2$.
			Applying $0$-dimensional unweighted \NIPH{} with scaling factor $s\le d_2/d_1$ in direction $\psi$ will then yield a multiplicative shift diagram with a peak at $s_1:=\sqrt{(s^2-1)\cos^2 (\psi-\varphi)+1}$ with weight $(n_1-1)n_2$ and a second peak at $s_2:=\sqrt{(s^2-1)\smash{\sin}^2 (\psi-\varphi)+1}$ with weight $n_2-1$.
			In particular, when scaling in direction $\varphi$ the maximum will appear at a value of $s$, while scaling in an orthogonal direction will produce a maximum at value $1$.
		\end{restatable}
	\end{adjustbox}
	\hfill%
	\begin{adjustbox}{valign=b, minipage=0.24 \linewidth}
		\includegraphics[width=\linewidth]{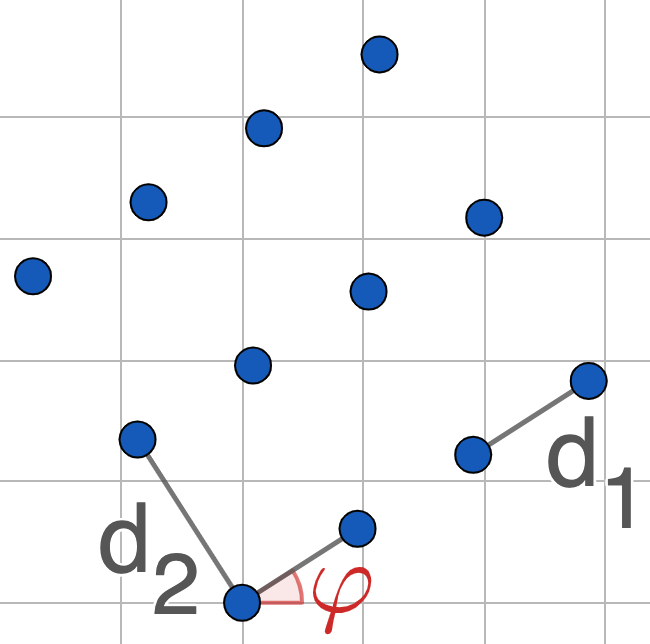}
	\end{adjustbox}
	
	\vincent{There is a version of this proof where each of the points can have additive noise $\le \varepsilon$ that is similarly easy to show. We might want to use this? Or at least include the statement of it?}
	%\michael{I think providing the statement below is fine --- if needed the proof can go into an appendix ;)}
	
	\begin{remark}
		This theorem explains how we can use the \NIPH{} framework to extract orientations of point clouds: By probing different scaling directions and the associated mult. shift diagrams, we can look for the diagrams with the largest $x$-value of the peak. This will then approximate the orientation of the point cloud.
		Guarantees in the setting with noise and the proofs can be found in \Cref{sec:proofs}.
	\end{remark}
	\section{Experiments}
	\label{sec:Experiments}
	\paragraph{Orientation, orientational variance, and scaling factor}
	\begin{figure}[b]
		\begin{adjustbox}{valign=b, minipage=0.48 \linewidth}
			\scriptsize
			\begin{tabular}{c|rrr}
				\toprule
				$V=\text{std}(\varphi )$&$\sqrt{\text{MSE}}$ $\varphi$& $\sqrt{\text{MSE}}$ scaling&$\sqrt{\text{MSE}}$ $\sqrt{V}$\\
				\midrule
				$0^\circ$& $0.3^\circ$& $0.09$&$0.26$\\
				$5^\circ$& $0.8^\circ$& $0.11$&$0.18$\\
				$10^\circ$& $1.6^\circ$& $0.15$&$0.12$\\
				$15^\circ$& $2.3^\circ$& $0.18$&$0.08$\\
				$20^\circ$& $3.0^\circ$& $0.19$&$0.07$\\
				$25^\circ$& $4.0^\circ$& $0.19$&$0.07$\\
				$30^\circ$& $5.1^\circ$& $0.19$&$0.08$\\
				$35^\circ$& $6.2^\circ$& $0.20$&$0.08$\\
				$40^\circ$& $8.0^\circ$& $0.20$&$0.08$\\
				\midrule
				\bottomrule
			\end{tabular}
			\caption{	\textbf{Quantitative performance of \NIPH{}.}
				We have run \NIPH{} on a point cloud sampled from 200 oriented rectangles with different orientational variances and $s=2$.
				We show the root of the mean squared error of the predictions%orientation ($\varphi$), orientational variance ($\text{std}(\varphi )$), and scaling factor as predicted by 
				of \NIPH{} depending on the orientational variance of $X$.}%vincent{Combine table with something to free up space!}}
		\label{tab:SynthExp}
	\end{adjustbox}%
	\hfill%
	\begin{adjustbox}{valign=b, minipage=0.42  \linewidth}
		\begin{center}
			\includegraphics[width=0.9\linewidth]{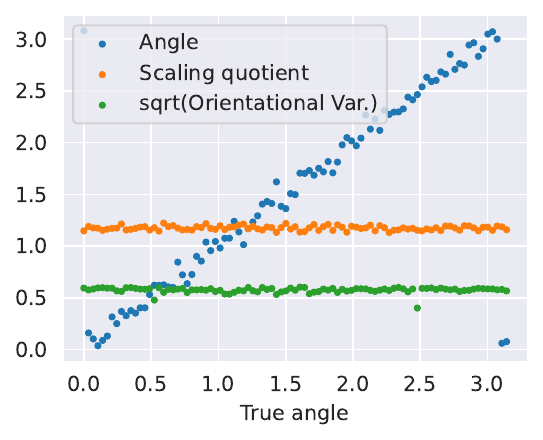}
			\caption{\textbf{Predictions of \NIPH{} in high orientational variance setting.} $x$-axis: True $\varphi$. 
				\emph{Blue:} $\varphi$.
				\emph{Green:} $\sqrt{V}$. (True value: $0.5$).
				\emph{Orange:} Scaling $s$. (True value: $1.5$).}
			\label{fig:SynthExpVaryingPhi}
		\end{center}
	\end{adjustbox}
\end{figure}

We ran experiments on synthetic data to verify that \NIPH{} can indeed infer information on the orientation, scaling, and variance inside a data set.
Given an angle $\varphi$, an orientational variance $V$, and a scaling factor $s$, we created a point cloud resembling ellipses or rectangles with sizes between $0.2$ and $2$.
We scaled the different axes of the ellipses/rectangles according to the scaling factor $s$, oriented them in the general direction of $\varphi$ and added independent normal distributed noise with variance $V$ and mean $0$ to each orientation of the individual ellipses/rectangles.

We then used the \NIPH{} pipeline to infer the parameters of the direction of orientation $\varphi$, the orientational variance $V$, and the scaling factor $s$ from the point cloud data.
%In \Cref{fig:SynthExpSketch}, we have highlighted the intermediate steps of \NIPH{}, i.e.\ the starting point cloud, the persistent homology death density diagrams, and the multiplicative shift diagrams.
We report the prediction accuracy of \NIPH{} in \Cref{tab:SynthExp} and \Cref{fig:SynthExpVaryingPhi}.
We note that even with high orientational variance in the synthetic data, \NIPH{} is able to predict orientation high accuracy, and obtains scaling and orientational variance with medium accuracy. \textsmaller{PCA} cannot infer any of these parameter, see \Cref{sec:pca}.

\paragraph{Road networks}
\begin{figure}[tb!]
	\begin{center}
		\begin{subfigure}{0.23\columnwidth}
			\includegraphics[width=\columnwidth]{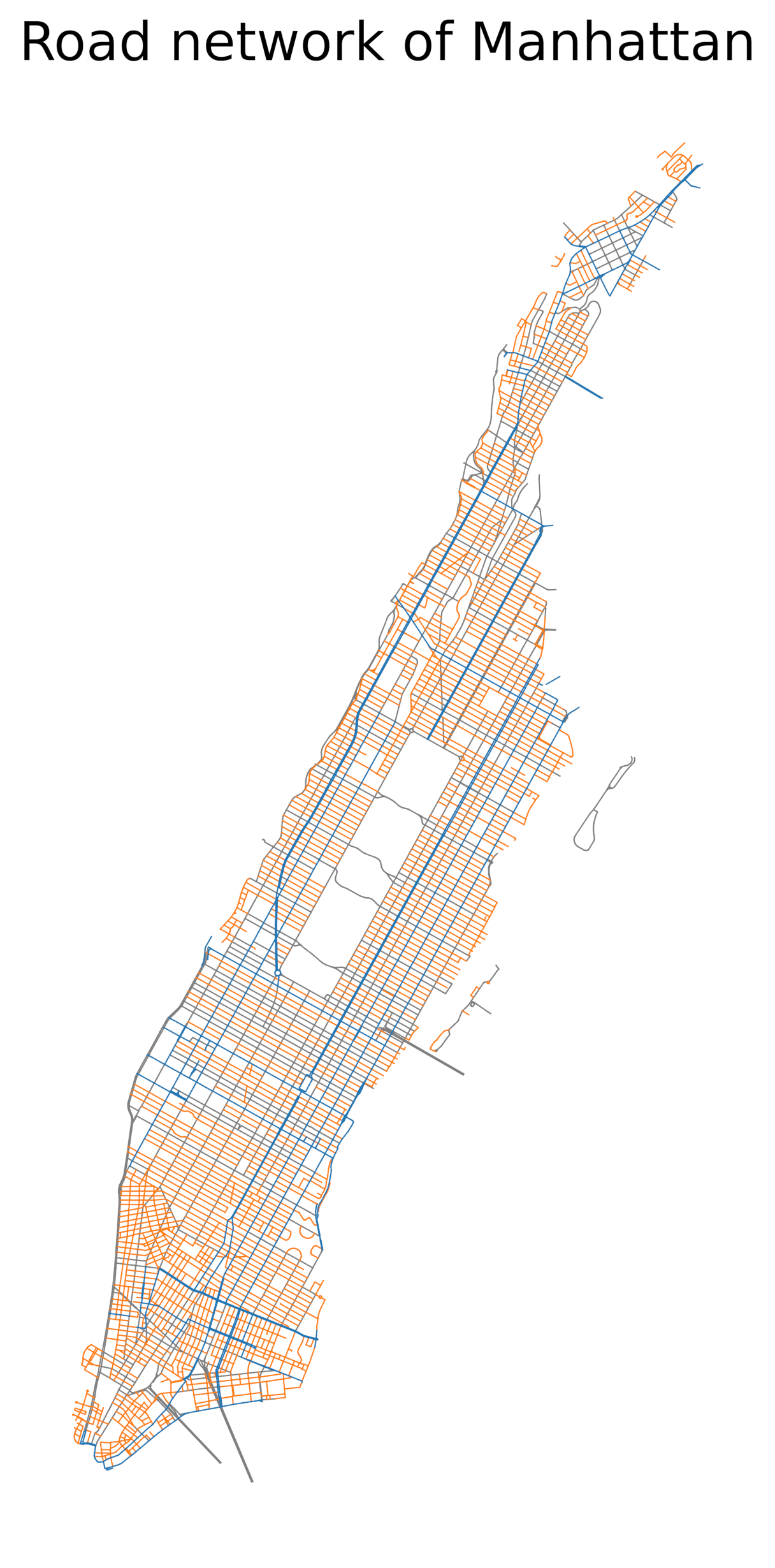}
		\end{subfigure}
		\begin{subfigure}{0.38\columnwidth}
			\includegraphics[width=\columnwidth]{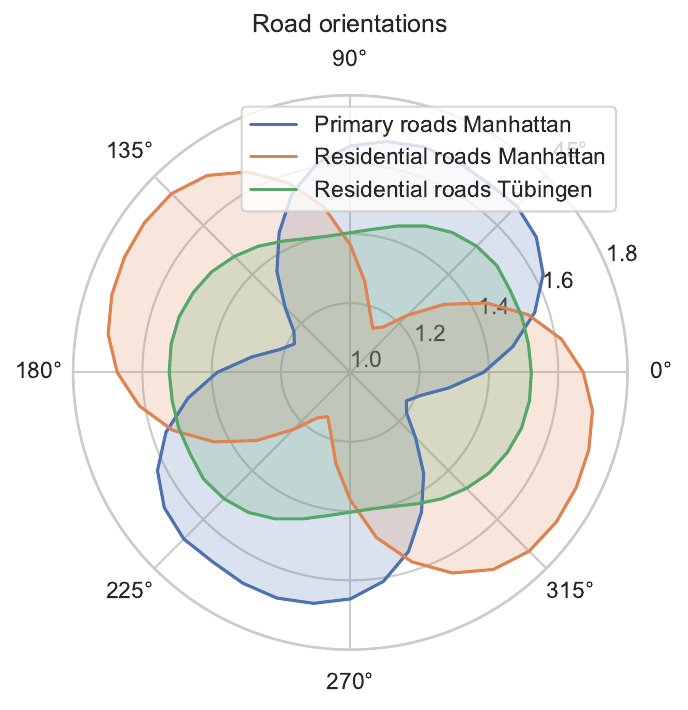}
		\end{subfigure}
		\begin{subfigure}{0.37\columnwidth}
			\includegraphics[width=\columnwidth]{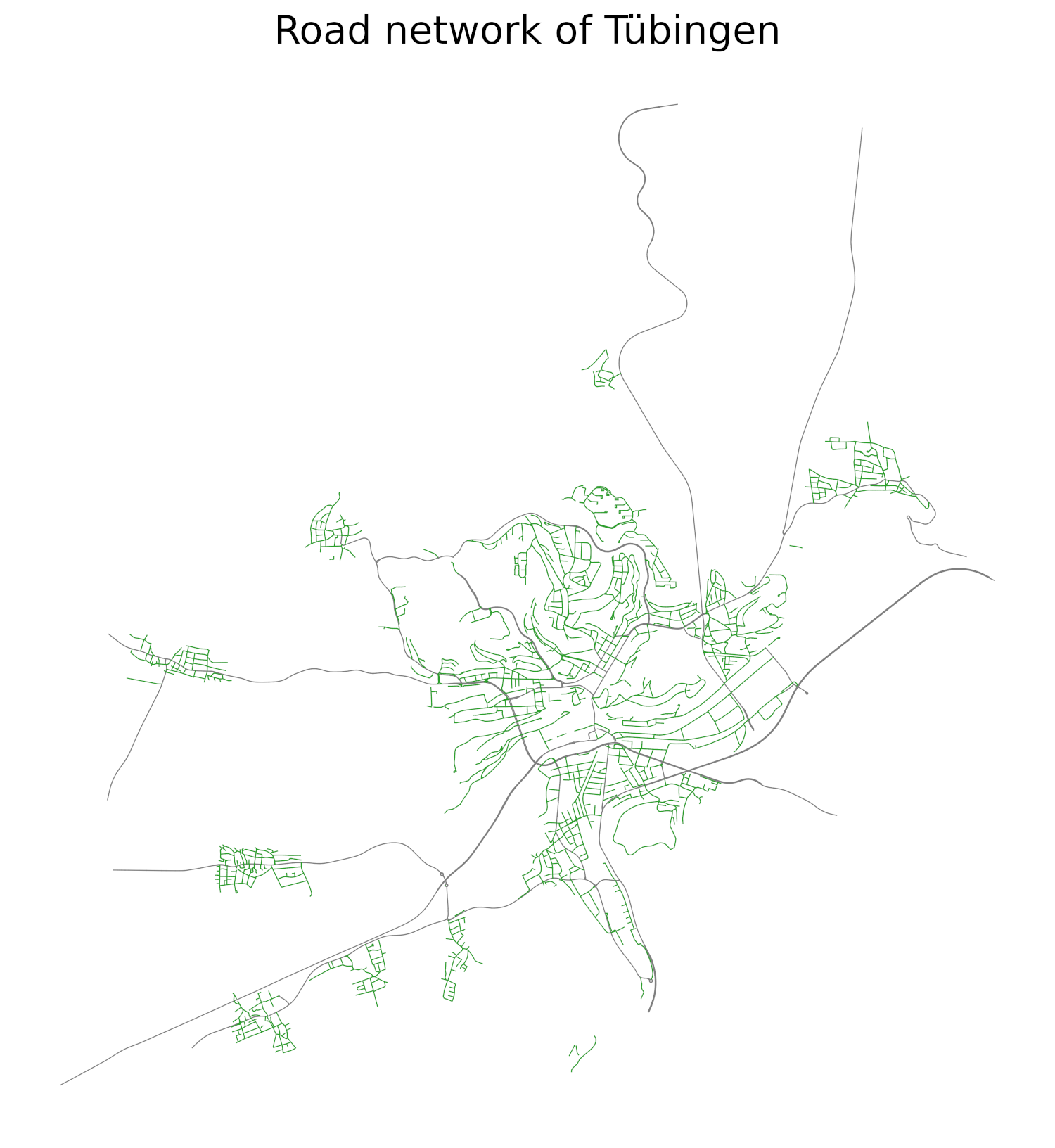}
		\end{subfigure}
		\caption{\emph{Left and right:} Road networks of Manhattan and Tübingen. Different colours denote different type of roads.
			\emph{Middle:} Strength of orientation along different directions of different road types in Manhattan and Tübingen. Residential roads (\enquote{streets}) and primary roads (\enquote{avenues}) have a very strong orientation in orthogonal directions. %Note, however, that not all roads adhere to this strict orientation.
			Residential roads in Tübingen have a very faint bias of an east--west orientation.
			The strength of orientation is measured by the $x$-value of the peaks in the corresponding mult.\ shift diagrams.
		}
		\label{fig:RoadNetworks}
	\end{center}
\end{figure}
\begin{figure}[tb!]
	\begin{center}
		\begin{subfigure}{0.37\columnwidth}
			\includegraphics[width=\columnwidth]{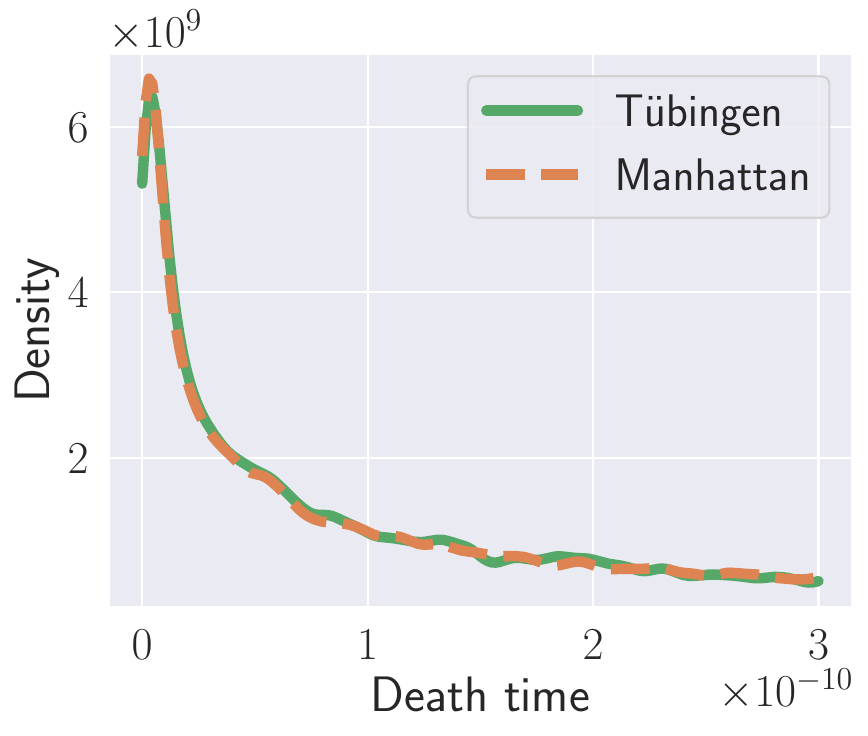}
		\end{subfigure}
		\begin{subfigure}{0.61\columnwidth}
			\includegraphics[width=\columnwidth]{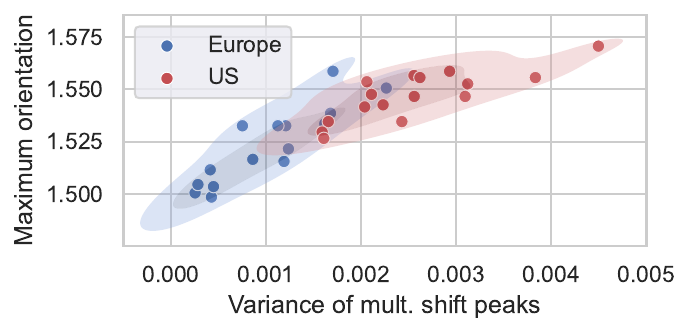}
		\end{subfigure}
		\caption{\emph{Left:} $0$-dimensional \PH{} density plot corr.\ to \Cref{fig:RoadNetworks}. Virtually no difference exists between the data sets exists, underscoring the relevance of \NIPH{}.
			\emph{Right:} $2$ measures of orientation strength of points clouds sampled on residential roads in $15$ small European and US cities.
			While outliers exist, the grid-like structure of US cities results in a higher strength of orientation in the point clouds.
		}
		\label{fig:ManyCities}
	\end{center}
\end{figure}
We used $0$-dimensional \NIPH{} on road network data from Manhattan, New York City, USA, and Tübingen, Germany.
We first extracted the road networks from openstreetmap data (\cite{OpenStreetMap}).
We then created a point cloud by choosing random locations of $130000$ cars (NYC) and $43000$ cars (Tübingen) on the respective road networks.
We then ran the \NIPH{}-pipeline for $0$-dimensional homology on the point clouds restricted to specific road types.
This allows us to extract information on the predominant orientations in the point cloud, see \Cref{fig:RoadNetworks}.
The point clouds generated from the Manhattan road network have a very strong orientation, whereas the point clouds resembling the Tübingen road network have virtually no preferred orientation.

Additionally, we used \NIPH{} on $15$ US and European cities to extract the orientational strength of point clouds sampled on the road network.
On average, US cities have a higher strength of orientation than their European counterparts (\Cref{fig:ManyCities} \emph{Right}).
This experiment demonstrates that \NIPH{} is a powerful tool to extract notions of orientation and other information inaccessible to ordinary \PH{}.
The $0$\textsuperscript{th} \PH{} density diagrams show virtually no difference between Tübingen and Manhattan (\Cref{fig:ManyCities} \emph{Left}).
%\Cref{fig:RoadNetworksMS} contains the multiplicative shift corresponding to \Cref{fig:RoadNetworks}.
\section{Discussion}
\paragraph{Limitations and future work}
Because persistent homology in higher dimensions is computationally expensive, for very large point clouds, only $0$-dimensional \NIPH{} is feasible.
However, landmark sampling or down-sampling can often mitigate this problem.
In certain cases, especially for inhomogeneous data sets, the optimal transport will produce an incorrect matching between the persistent homology classes, falsifying the results of the algorithm (See \Cref{sec:limitations} for an in-depth analysis).
Future work will consider other methods to match homology classes, more deeply rooted in the algebraic machinery behind \PH{}, including tools like vineyards \citep{cohensteiner2006}.
Furthermore, using the orientations of the death simplices (Cf. \Cref{subsec:AdditionalApproaches}) does not even require the computation of a matching.
The flexibility of the \NIPH{} framework allows for a whole avenue of research using different hand-crafted families of metrics for exciting new applications.
\paragraph{Conclusion}
We have introduced \NIPH{}, a novel method building on top of persistent homology.
\NIPH{} extracts additional topological and geometrical information by varying the distance function on the underlying space and analysing the corresponding shifts in the persistence diagrams.
We have verified the performance of \NIPH{} on a synthetic data set and on real-world road network data.
%We have discussed one specific way based on scalings of the base point cloud this general framework can be utilised and verified the performance of this approach on a synthetic data set and on real-world road network data.
%However, the basic framework of extracting information by varying a metric are far more general and open up an avenue of future work utilising different families of metrics and of ways of comparing the persistence diagrams.
%\vincent{\NIPH{} is good and you should use it. Kurz oder gar keine.}

\bibliographystyle{unsrtnat}
\bibliography{refpc.bib}
\appendix
\section{Proofs and Theoretical Guarantees for Noisy Grids}
\label{sec:proofs}
We will now give a proof of \Cref{thm:cleangrid} and \Cref{thm:noisygrid} proving guarantees for noisy grids:

\TheoremGuarantees*
\begin{proof}
	The \NIPH{} pipeline consists of multiple steps: First we will compute the persistent $0$-homology of the base space.
	Because of the grid structure, we will obtain $(n_1-1)n_2$ homology classes with a death time of $d_1$, $n_2-1$ classes with a death time of $d_2$ and $1$ class with death time $\infty$.
	Now we consider the scaled point cloud.:
	After rotating by $-\psi$, some geometric calculations reveal a scaling factor of
	\[
	s_1:=\sqrt{s^2\cos^2 (\psi-\varphi)+\smash{\sin}^2 (\psi-\varphi)}=\sqrt{(s^2-1)\cos^2 (\psi-\varphi)+1}\le s\le d_2/d_1
	\]
	in the $d_1$ direction and of $s_2:=\sqrt{(s^2-1)\smash{\sin}^2 (\psi-\varphi)+1}\ge 1$ in the $d_2$ direction.
	We notice that because of our choice of $s$, we still have that $s_1d_1\le s_2d_2$.
	Thus, the persistence diagram of the scaled point cloud consists of $(n_1-1)n_2$ classes with a death time of $s_1d_1$, $n_2-1$ classes with a death time of $s_2d_2$, and one class with death time $\infty$.
	Now, unweighted optimal transport in the $1d$ setting has a straight-forward solution inducing a $1:1$ mapping of the $(n_1-1)n_2$ points from $d_1$ to $s_1d_1$ with a multiplicative shift of $s_1$ and a $1:1$ mapping of the $n_2-1$ points from $d_2$ to $s_2d_2$ with a mult.\ shift of $s_2$.
	This then induces the multiplicative shift diagram as described in the theorem.
\end{proof}

%\vincent{
	\begin{restatable}[Guarantees for noisy grids]{theorem}{TheoremNoisyGuarantees}
		\label{thm:noisygrid}
		Assume we are in the situation of \Cref{thm:cleangrid} but with an additive random noise $\varepsilon$ with $|\varepsilon|<\delta$ added to every point independently, with $d_1+\delta<d_2-\delta$ and $s\le (d_2-2\delta)/(d_1+2\delta)$.
		The mult.\ shift diagram will then have a mass of $(n_1-1)n_2$ in the interval $[(s_1d_1-2s\delta)/(d_1+2\delta),(s_1d_1+2s\delta)/(d_1-2\delta)]$.
	\end{restatable}
	\begin{proof}
		The proof of the theorem in the noisy case works essentially the same as in the noise-free case.
		The important observation is that before the shift, both points vary at most by a distance of $\delta$ from their original grid position, whereas after the shift, both points vary at most by a distance of $s\delta$.
		Dividing these bounds yields a new approximation of the multiplicative shift.
	\end{proof}
	\section{Comparison to Principal Component Analysis}
	\label{sec:pca}
	\begin{figure}[tb!]
		\begin{center}
			\begin{subfigure}{0.24\columnwidth}
				\includegraphics[width=\columnwidth]{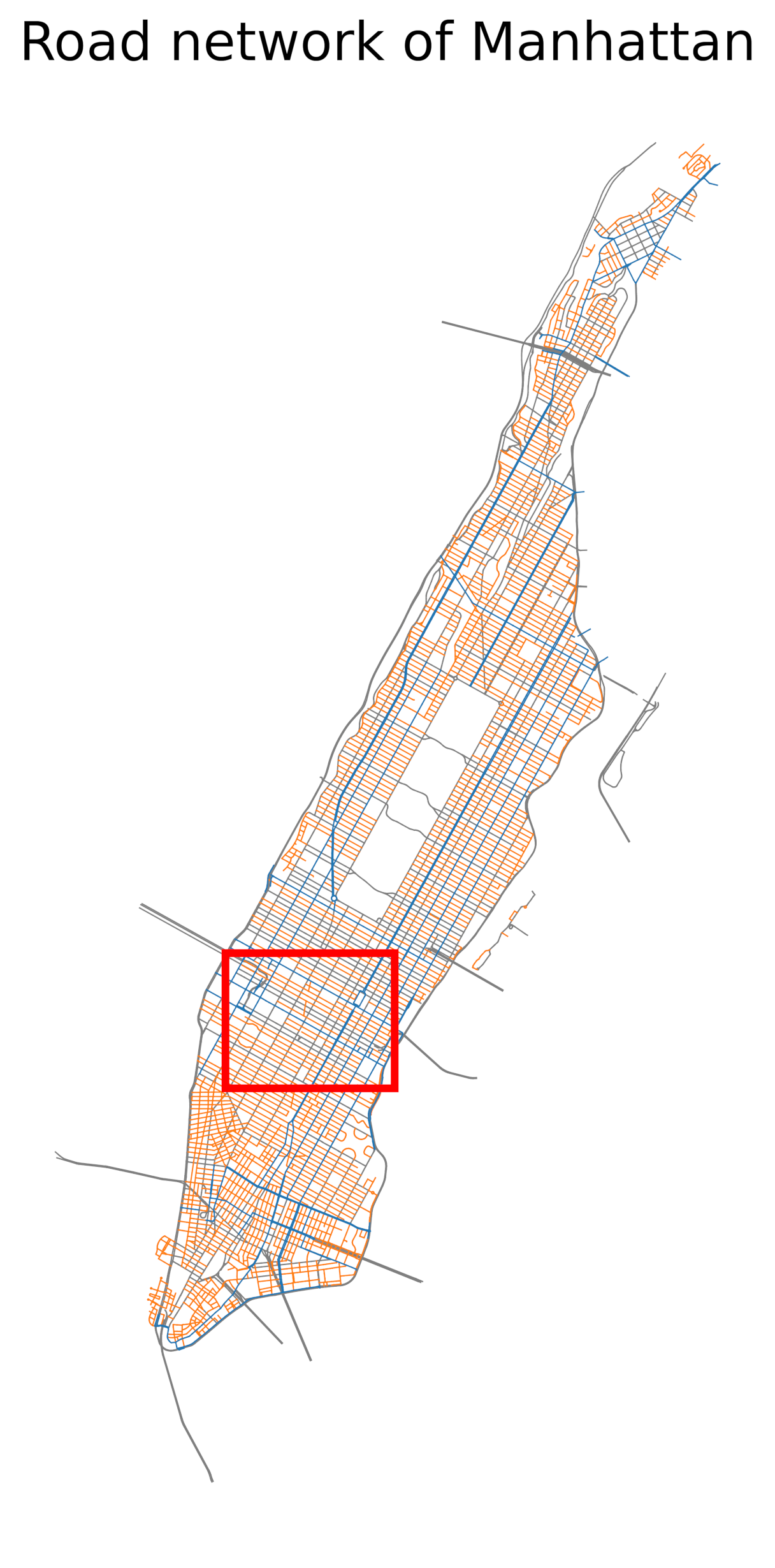}
			\end{subfigure}
			\begin{subfigure}{0.4\columnwidth}
				\includegraphics[width=\columnwidth]{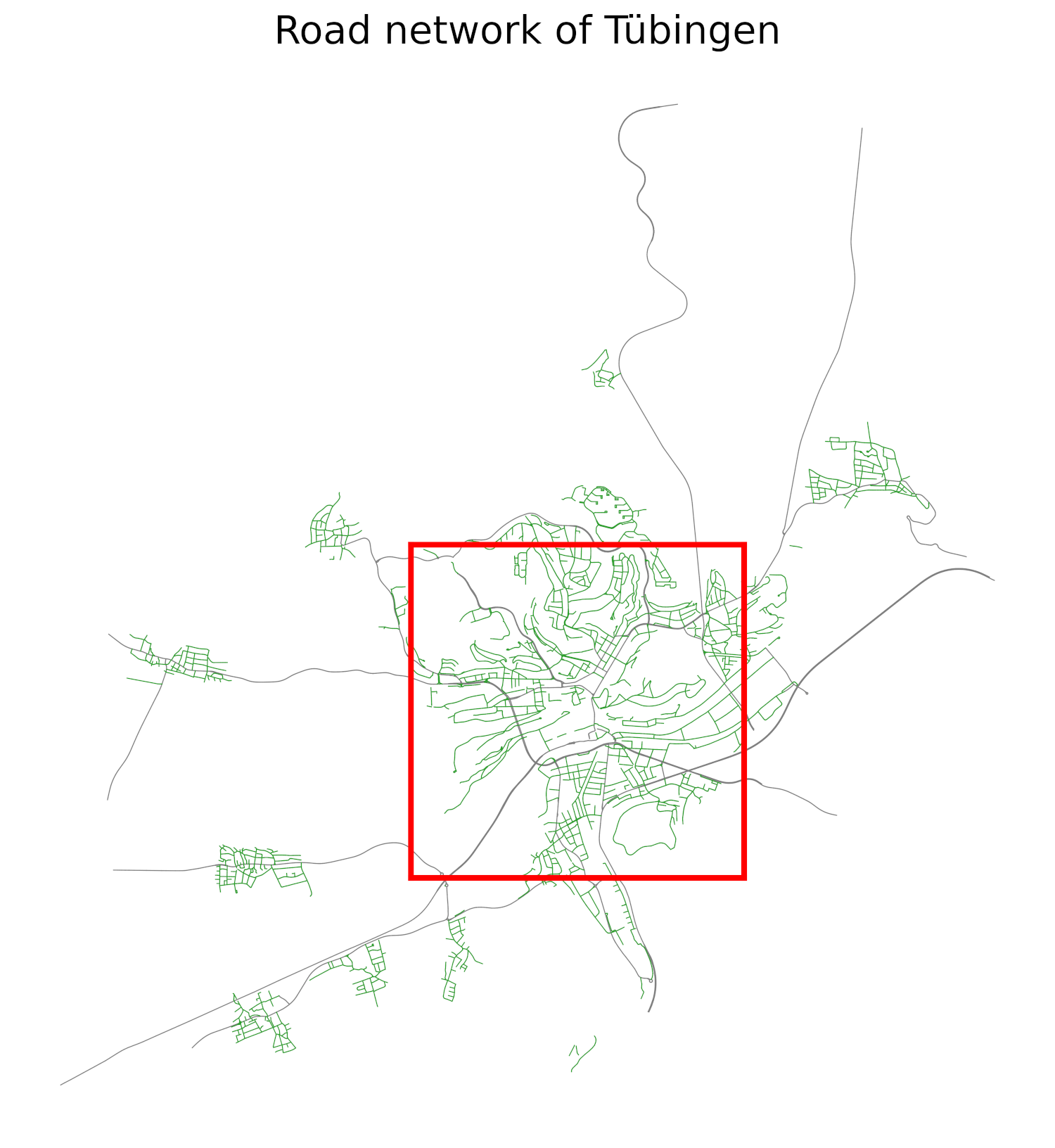}
			\end{subfigure}
			\begin{subfigure}{0.48\columnwidth}
				\includegraphics[width=\columnwidth]{figs/Road_orientations.pdf}
			\end{subfigure}
			\begin{subfigure}{0.48\columnwidth}
				\includegraphics[width=\columnwidth]{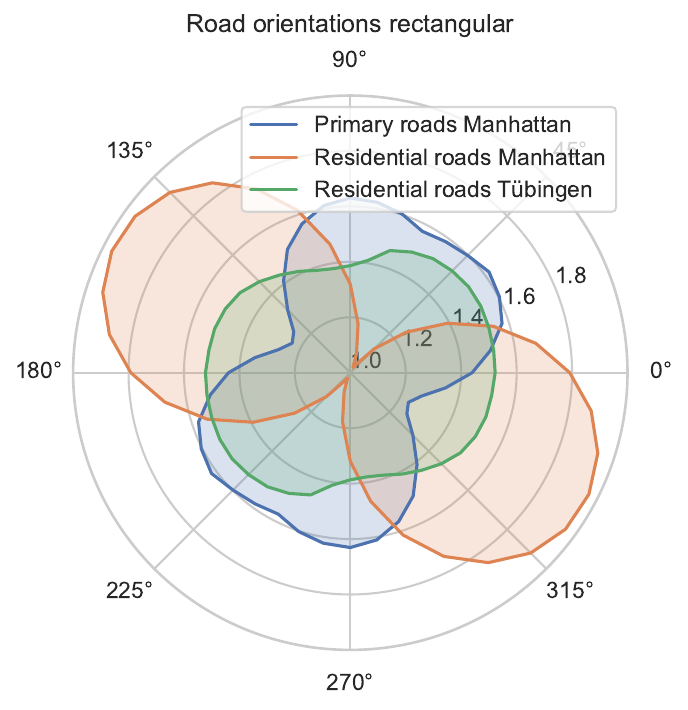}
			\end{subfigure}
			\begin{subfigure}{0.5\columnwidth}
				\scriptsize
				\begin{tabular}{c|rr}
					\toprule
					\textbf{\textsmaller{PCA} results}& $\varphi$ everything &$\varphi$ rectangle\\
					\midrule
					Primary roads Manhattan& $68.8^\circ$& $154.4^\circ$\\
					Residential roads Manhattan& $67.8^\circ$& $26.4^\circ$\\
					Residential roads Tübingen& $47.5^\circ$&$73.4^\circ$\\
					\midrule
					\bottomrule
				\end{tabular}
			\end{subfigure}
			
			\caption{\emph{Top:} Road networks of Manhattan and Tübingen. Different colours denote different type of roads. Red rectangle denotes borders of smaller dataset
				\emph{Middle:} Strength of orientation along different directions of different road types in Manhattan and Tübingen. Residential roads (\enquote{streets}) and primary roads (\enquote{avenues}) have a very strong orientation in orthogonal directions. \emph{Left:} Entire dataset, \emph{Right:} Dataset restricted to rectangle.
				Note that \NIPH can extract the same orientations on both datasets.
				\emph{Bottom:} Angles of principal directions obtained by \textsmaller{PCA}.
				Although the local structure is unchanged when restricting to the rectangle, the resulting \textsmaller{PCA} direction is entirely different.
				Note furthermore \textsmaller{PCA} estimates primary and residential roads in Manhattan to have the same orientation which is entirely false.
			}
			\label{fig:RoadNetworkspca}
		\end{center}
	\end{figure}
	\begin{figure}[htb!]
		\begin{center}
			\begin{subfigure}{0.48\columnwidth}
				\includegraphics[width=\columnwidth]{figs/Different_Cities.pdf}
			\end{subfigure}
			\begin{subfigure}{0.48\columnwidth}
				\includegraphics[width=\columnwidth]{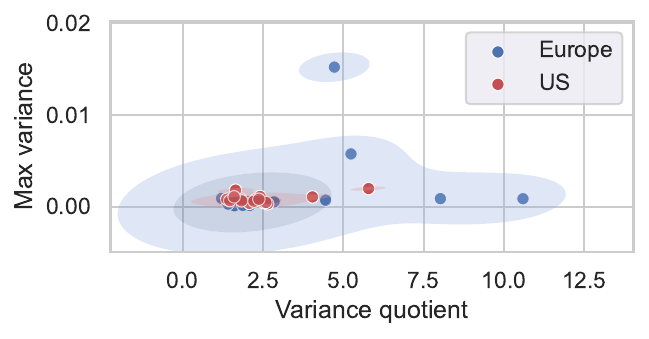}
			\end{subfigure}
			\caption{Orientational strength of road grids in US and European cities extracted by \emph{Left:} \NIPH{} and \emph{Right:} \textsmaller{PCA}.
			}
			\label{fig:ManyCitiespca}
		\end{center}
	\end{figure}
	\begin{figure}[tb!]
		\begin{center}
			\begin{subfigure}{0.47\columnwidth}
				\includegraphics[width=\columnwidth]{figs/synth_var_phi.pdf}
			\end{subfigure}
			\begin{subfigure}{0.47\columnwidth}
				\includegraphics[width=\columnwidth]{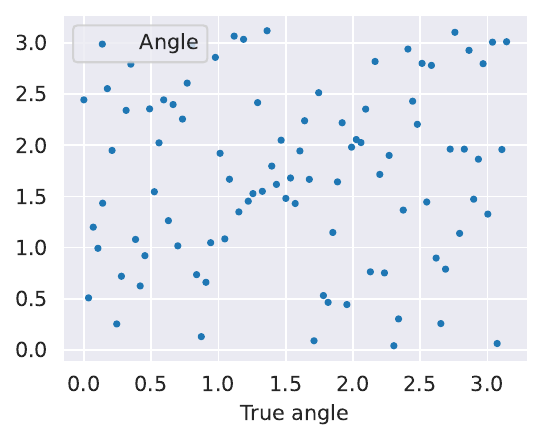}
			\end{subfigure}
			\caption{Trying to recover the orientation of synthetic data sets. \emph{$x$-axis:} True orientation. \emph{Blue on the $y$-axis:} Inferred orientation. \emph{Left:} Results obtained with \NIPH{}, \emph{Right:} Results obtained with \textsmaller{PCA}. It is clear that \NIPH{} outperforms \textsmaller{PCA} by a wide margin.
			}
			\label{fig:SynthExpPca}
		\end{center}
	\end{figure}
	Principal component analysis (\textsmaller{PCA}) is a standard tool in machine learning and data science to determine the principal directions in high-dimensional point clouds.
	However, \textsmaller{PCA} only considers the global shape of the data, whereas \NIPH{} is able to extract the orientation based on local structures, which is important in many real-world applications.
	When using \textsmaller{PCA} on the data sets introduced in \Cref{sec:Experiments}, we noticed that the principle direction inferred by \textsmaller{PCA} depends on the shape of the considered subset/cutout of the data and not on the local structures in the data.
	In \Cref{fig:RoadNetworkspca} we show that \textsmaller{PCA} does not infer correct orientations in the Manhatten/Tübingen road networks.
	In \Cref{fig:ManyCitiespca} we have tried to replicate the results of \Cref{fig:ManyCities} using \textsmaller{PCA}. We try to determine the orientational variance by the maximum variance of the data set in any direction ($y$-axis) and the quotient of the variance of the data in the first and second principal direction.
	As shown in the diagram, \textsmaller{PCA} fails to capture meaningful information.
	In \Cref{fig:SynthExpPca} we have compared the effectiveness of \NIPH{} with the effectiveness of \textsmaller{PCA} in uncovering hidden $1$-dimensional orientations on synthetic datasets.
	\NIPH{} clearly outperforms \textsmaller{PCA}.
	\FloatBarrier
	\section{Limitations}
	\label{sec:limitations}
	\begin{figure}[tb!]
		\begin{center}
			\begin{subfigure}{0.47\columnwidth}
				\includegraphics[width=\columnwidth]{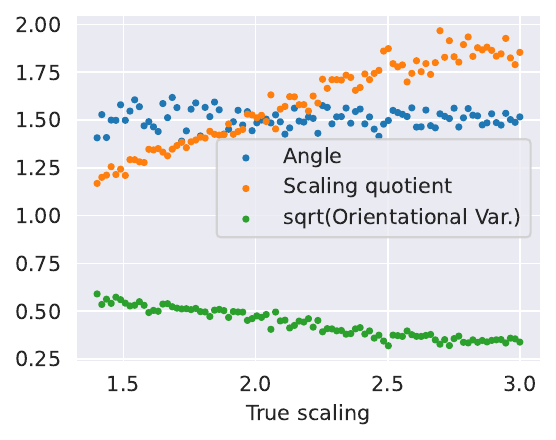}
			\end{subfigure}
			\begin{subfigure}{0.47\columnwidth}
				\includegraphics[width=\columnwidth]{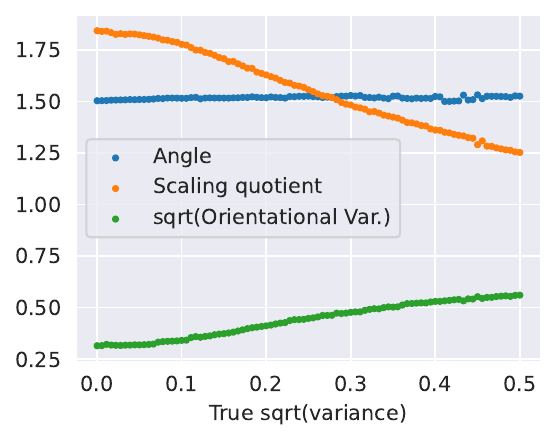}
			\end{subfigure}
			\caption{\emph{Left:} Estimations of \NIPH{} for fixed orientation angle $\phi=1.5$, orientational variance $\sqrt{V}=0.2$ and varying scaling factor.
				\emph{Right:} Estimation of \NIPH{} for fixed orientation angle $\phi=1.5$, scaling factor $s=2$ and varying orientational variance.
				\NIPH{} systematically overestimates the orientational variance and underestimates the scaling factor. This effect is more pronounced for high orientational variances and low scaling factors.
			}
			\label{fig:Limitations}
		\end{center}
	\end{figure}
	\begin{figure}[tb!]
		\begin{center}
			\begin{subfigure}{0.47\columnwidth}
				\includegraphics[width=\columnwidth]{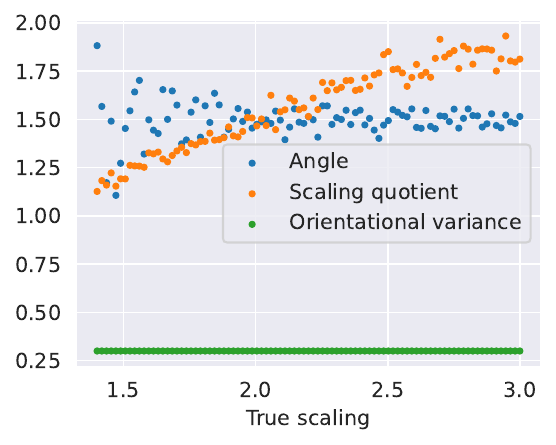}
			\end{subfigure}
			\begin{subfigure}{0.47\columnwidth}
				\includegraphics[width=\columnwidth]{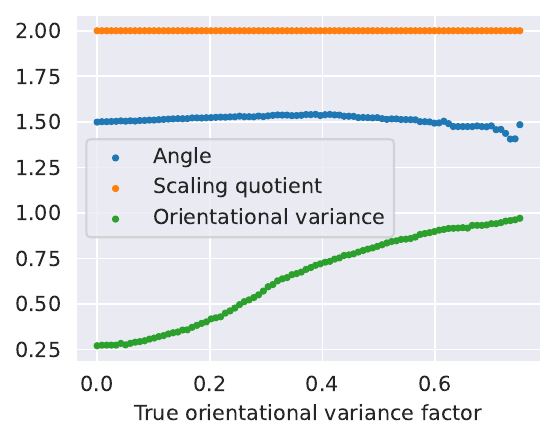}
			\end{subfigure}
			\caption{Performance of \NIPH{} when being told one correct parameter. \emph{Left:} Estimations of \NIPH{} for fixed orientation angle $\phi=1.5$,  fixed \emph{known} orientational variance $\sqrt{V}=0.3$ and varying scaling factor.
				\emph{Right:} Estimation of \NIPH{} for fixed orientation angle $\phi=1.5$, \emph{known} scaling factor $s=2$ and varying orientational variance.
			}
			\label{fig:LimitationsMoreKnowledge}
		\end{center}
	\end{figure}
	Across all our experiments, \NIPH{} was able to extract the underlying orientation of the point cloud with high accuracy.
	However, in certain parameter regimes, \NIPH{} systematically underestimated the scaling factor and overestimated the orientational variance, see \Cref{fig:Limitations}.
	However, because of the systematic nature of the error, this may be mitigated by training an \textsmaller{ML} model on the \NIPH{} outputs to infer the true parameters.
	We claim that there are two separate effects that likely contribute to these results:
	\begin{enumerate}[i)]
		\item Firstly, both increasing the orientational variance and decreasing the scaling factor results in decreasing the $x$-value of the peaks of the multiplicative death shift density diagrams sampled in the direction of the primary orientation of the data set.
		For the scaling factor, this is simply the case because a small scaling factor limits the maximum possible multiplicative shift.
		For the orientational variance, this is the case because a higher orientational variance means that fewer of the ellipses perfectly align with the sampling direction, thus decreasing the maximal possible multiplicative shift.
		Although low scaling factor and high orientational variance induce different behaviour in the other peaks, the effects on the peaks with the largest $x$-values are the most pronounced and thus are the main contributors to the parameter estimation of \NIPH.
		\item The second effect leading to this over- and underestimation is the noise introduced to the optimal transport maps due to the varying size of the ellipses, with the smaller principal axis ranging from  a length of $0.2$ to $2$.
		Because the cost of the optimal transport assignment increases with higher multiplicative shifts, this likely decreases the $x$-value of the peaks of the multiplicative death shift density diagrams.
		Because we have no mathematical model to factor in this effect, the optimisation algorithm tries to approximate this effect by increasing the orientational variance.
		(In our experiments by roughly $0.25$, see \Cref{fig:LimitationsMoreKnowledge}.)\vincent{I could do experiments by varying this variance of ellipse size and study the effect.}
	\end{enumerate}
	
	To further study these two effects, we have analysed the performance of \NIPH{} if the optimisation algorithm gets told one correct parameter.
	Because \NIPH{} infers the correct orientation with high accuracy even in the base case, we have implemented this analysis only for the scaling factor and the orientational variance, see \Cref{fig:LimitationsMoreKnowledge}.
	It appears that providing \NIPH{} with the correct orientational variance does not improve the performance, but even worsens it (\emph{Left} diagram).
	This supports the second claim: Because \NIPH{} now cannot model some of the noise as noise stemming from orientation variance, its predictive performance is decreased.
	The \emph{right} diagram shows that giving \NIPH{} the correct scaling factor does not decrease performance. However, now the orientational variance is roughly constantly overestimated by a value of $0.25$.
	This supports the first claim: In the previous experiments (\Cref{fig:Limitations} \emph{Right}), overestimation of the orientational variance decreased, as the underestimation of the scaling quotient increased.
	When preventing underestimation of the scaling quotient, the orientational variance now is consequently overestimated.
	
	Future work could either work on training an \textsmaller{ML} approach to compensate for these systematic errors, or to give a more complex mathematical model that better incorporates all sources of noise.
	However, the main goal of the current paper is to introduce the novel idea of changing the underlying metric of point clouds and then quantifying changes in the \PH{} to infer geometric information on the point cloud and to provide a first proof-of-concept.
	We have demonstrated that \NIPH{} can harvest rich information, can accurately infer the orientation of a point cloud, and can give a sensible estimation of the scaling factor and orientational variance.
	We believe that this already serves as a  proof-of-concept of the powerful ideas of \NIPH{}, leaving some fine-tuning on implementational details to future work.
	\FloatBarrier
	\section{Implementational details}
	We implemented \NIPH{} using python.
	The code to replicate all experiments in this paper can be found here \href{https://git.rwth-aachen.de/netsci/publication-2023-non-isotropic-persistent-homology}{https://git.rwth-aachen.de/netsci/publication-2023-non-isotropic-persistent-homology}.
	We generated the density plots using gaussian kernel densities estimates and set the bandwidth parameter using Scott's rule.
	For solving the optimisation problem, we used dual annealing from the SciPy optimisation library.
	We determined the peaks of the multiplicative death shift density curves using a grid search.
	For \Cref{fig:ManyCities} \emph{Right}, we compared point clouds sampled from the residential road networks of $\sim 15$ US and European cities.
	The cities were chosen based on a size of roughly $50000$--$100000$ inhabitants and city boundaries correctly represented by openstreetmaps.
	We used the European Cities Tübingen,
	Göttingen,
	Cambridge UK,
	Meran,
	Södertälje,
	Belfort,
	Namur,
	Delft,
	Quedlinburg,
	Passau,
	Jena,
	Legnica,
	Flensburg,
	Görlitz, and
	King's Lynn and West Norfolk.
	For US cities, we used
	Albany NY,
	South Bend IN,
	Boulder CO,
	Davenport IA,
	Cupertino CA,
	Nampa ID,
	Cheyenne WY
	Wheaton IL,
	Utica NY,
	Santa Fe NM,
	Bismarck ND,
	Coon Rapids MN,
	Sioux City IA,
	Great Falls MT,
	Rapid City SD,
	and
	Sunrise FL.
	For \NIPH, we sampled $15$ different directions and $9$ different scaling factors between $1.2$ and $2.5$.
	We have sampled the locations of the cars uniformly and independently at random across the entire road network.
	Hence the expected number of cars on a certain road segment across a city only depended on the length of the segment.
	
	For the synthetic experiments as shown in \Cref{fig:SynthExpVaryingPhi}, for each dataset we have generated $200$ ellipses with random centre point in a $3000\times3000$ square with each 100 points. We controlled the principal direction of the ellipses and added normal distributed noise to this principal direction based on the given orientational variance.
	We uniformly set the length of the smaller principal axis to a value between $0.2$ and $2$ and scaled the larger principal axis accordingly.
	An excerpt of the datasets can be found in the top of \Cref{fig:SynthExpSketch}.
	\subsection{Analytic expressions for expected multiplicative shift}
	\label{sec:AnalyticDescription}
	For $1$-dimensional homology, the expected multiplicative shift for a true orientation $\phi$, sampling orientation $\psi$, true scaling factor $s$, sampling scaling factor $S$ and orientational variance $V=0$ is the minimum of the shift in the scaling direction and the shift orthogonal to the scaling direction
	\[
	\expP_{\psi,S}(\varphi, 0, s):=\min \left(\frac{S}{\sqrt{
			1 + (S^2 - 1)\cdot \cos^2(\varphi-\psi)
	}},
	\frac{sS}{\sqrt{
			1 + (S^2 - 1)\cdot \sin^2(\varphi-\psi)		
	}}\right).
	\]
	
	For non-zero orientational variance $V$, we have to fold this function:
	\[
	\expP_{\psi,S}(\varphi, V, s) =
	\operatorname{argmax}_x \frac{1}{\sqrt{V2\pi}} \int_{-\infty}^\infty e^{-t^2/(2V)}\delta_{x=\expP_{\psi,S}(\varphi+t, 0, s)}dt
	\]
	We can approximate this using the weighted integral of the peaks:
	\[
	\expP_{\psi,S}(\varphi, V, s) =
	\frac{1}{\sqrt{V2\pi}} \int_{-\infty}^\infty e^{-t^2/(2V)}\expP_{\psi,S}(\varphi +t, 0, s)dt.
	\]
	In the $0$-dimensional case, we do the analogue for the following expected shift function:
	\[
	\expP_{\psi,S}(\varphi, 0, s):=\min \left(\sqrt{S^2 \cos^2	 (\varphi-\psi)^2 + \sin^2 (\varphi-\psi)},
	s\right).
	\]
	\FloatBarrier
	\section{Selection of Probing Directions and Scaling Factors}
	\label{sec:SelectionProbingDirections}
	\begin{figure}[tb!]
		\begin{center}
			\begin{subfigure}{0.30\columnwidth}
				\includegraphics[width=\columnwidth]{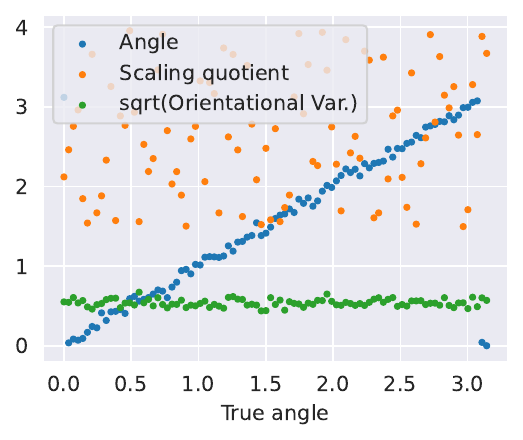}
			\end{subfigure}
			\begin{subfigure}{0.30\columnwidth}
				\includegraphics[width=\columnwidth]{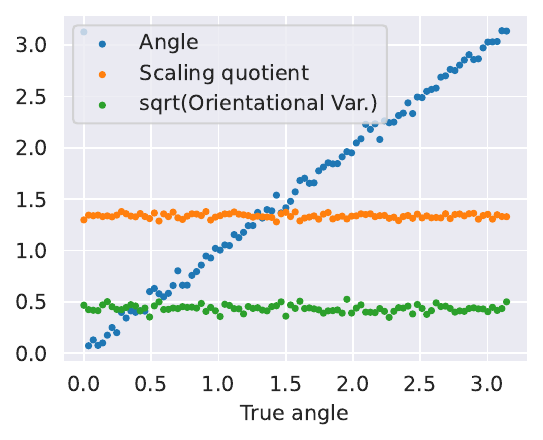}
			\end{subfigure}
			\begin{subfigure}{0.30\columnwidth}
				\includegraphics[width=\columnwidth]{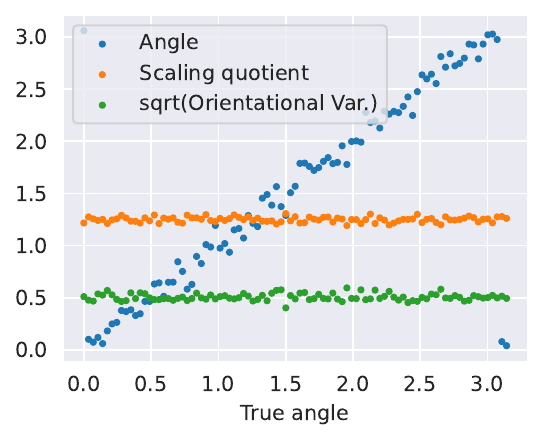}
			\end{subfigure}
			\begin{subfigure}{0.30\columnwidth}
				\includegraphics[width=\columnwidth]{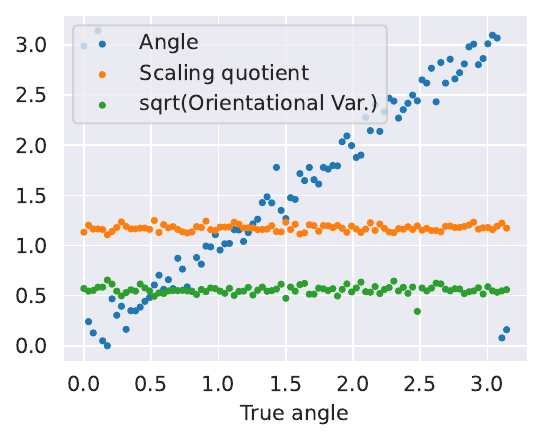}
			\end{subfigure}
			\begin{subfigure}{0.30\columnwidth}
				\includegraphics[width=\columnwidth]{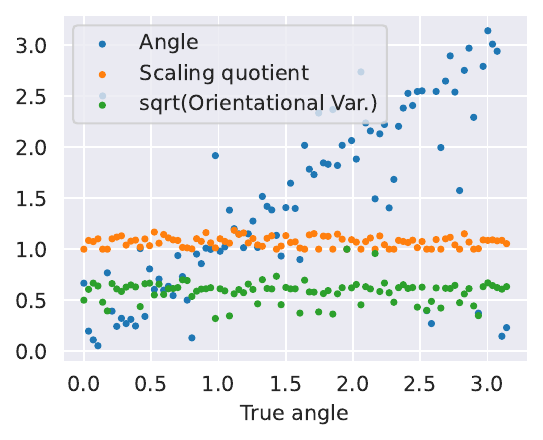}
			\end{subfigure}
			\begin{subfigure}{0.30\columnwidth}
				\includegraphics[width=\columnwidth]{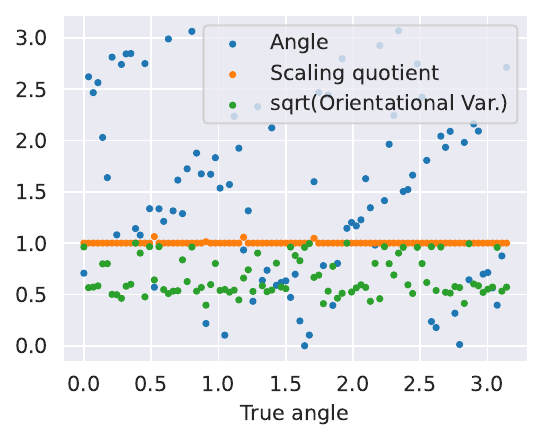}
			\end{subfigure}
			\caption{\textbf{Performance of \NIPH{} for different probing scaling factors and $8$ probing directions.} We have a ground truth orientational standard deviation of $0.5$, scaling factor of $1.5$ and varying angle. \emph{Top left:} Probing scaling factor $1.5$, \emph{Top centre:} Probing scaling factor $2.0$, \emph{Top right:} Probing scaling factor $2.5$, \emph{Bottom left:} Probing scaling factor $3.0$, \emph{Bottom centre:} Probing scaling factor $4.0$, \emph{Bottom right:} Probing scaling factor of $8.0$.
			}
			\label{fig:DifferentProbingScales}
		\end{center}
	\end{figure}
	\begin{figure}[tb!]
		\begin{center}
			\begin{subfigure}{0.30\columnwidth}
				\includegraphics[width=\columnwidth]{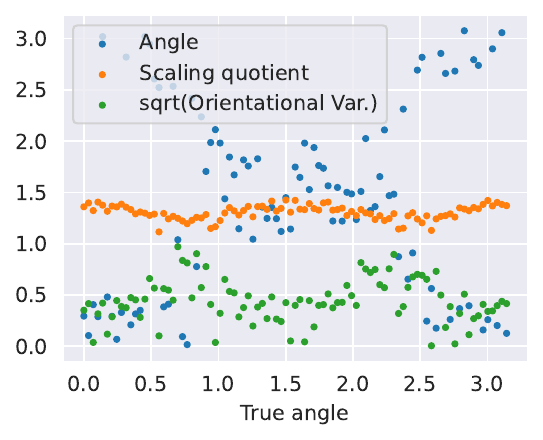}
			\end{subfigure}
			\begin{subfigure}{0.30\columnwidth}
				\includegraphics[width=\columnwidth]{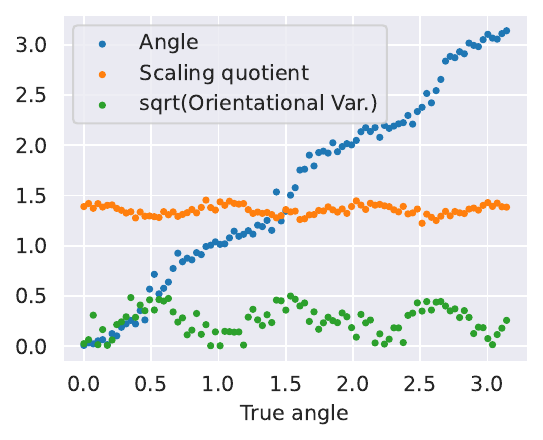}
			\end{subfigure}
			\begin{subfigure}{0.30\columnwidth}
				\includegraphics[width=\columnwidth]{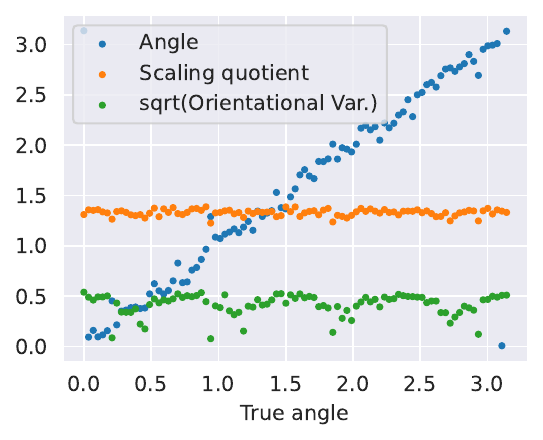}
			\end{subfigure}
			\begin{subfigure}{0.30\columnwidth}
				\includegraphics[width=\columnwidth]{figs/synth_var_phi_8angles.pdf}
			\end{subfigure}
			\begin{subfigure}{0.30\columnwidth}
				\includegraphics[width=\columnwidth]{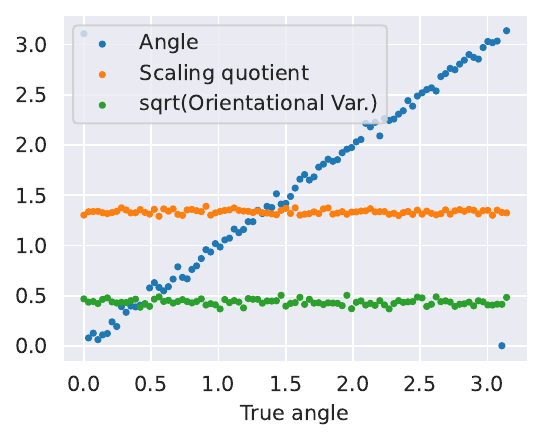}
			\end{subfigure}
			\begin{subfigure}{0.30\columnwidth}
				\includegraphics[width=\columnwidth]{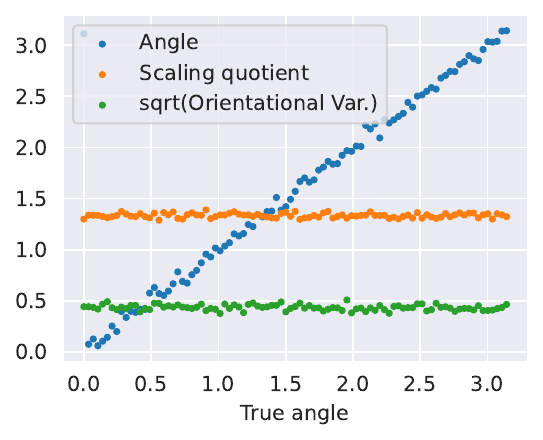}
			\end{subfigure}
			\caption{\textbf{Performance of \NIPH{} for different probing scaling factors and $8$ probing directions.} We have a ground truth orientational standard deviation of $0.5$, scaling factor of $1.5$ and varying angle. Numbers of probing angles from left to right \emph{first Row:} $2$--$3$--$4$, \emph{second Row:} $8$--$16$--$32$.
			}
			\label{fig:AngleCounts}
		\end{center}
	\end{figure}
	\begin{figure}[ht]
		\begin{center}
			\begin{tabular}{lrrr}
				\toprule
				Dataset&$\sqrt{\text{MSE}}$ $\varphi$& $\sqrt{\text{MSE}}$ scaling&$\sqrt{\text{MSE}}$ $\sqrt{V}$\\
				\midrule
				$\#2$ probing directions& $35^\circ$& $0.21$&$13^\circ$\\
				$\#3$ probing directions& $6.3^\circ$& $\mathbf{0.16}$&$16^\circ$\\
				$\#4$ probing directions& $4.7^\circ$ &$\mathbf{0.17}$&$7.7^\circ$\\
				$\#8$ probing directions& $2.9^\circ$& $\mathbf{0.17}$&$4.3^\circ$\\
				$\#16$ probing directions& $\mathbf{2.6^\circ}$& $\mathbf{0.17}$&$4.1^\circ$\\
				$\#32$ probing directions& $\mathbf{2.6^\circ}$& $\mathbf{0.17}$&$4.0^\circ$\\
				\midrule
				$1.5$ probing scaling factor& $3.0^\circ$& $1.4$&$3.5^\circ$\\
				$2$ probing scaling factor& $2.9^\circ$& $\mathbf{0.17}$&$4.3^\circ$\\
				$2.5$ probing scaling factor& $5.0^\circ$& $0.25$&$\mathbf{1.9^\circ}$\\
				$3$ probing scaling factor& $6.8^\circ$& $0.33$&$4.0^\circ$\\
				$4$ probing scaling factor& $22^\circ$& $0.43$&$7.9^\circ$\\
				$8$ probing scaling factor& $47^\circ$& $0.50$&$14^\circ$\\
				\midrule
				Combined scaling factors&$3.8^\circ$& $0.32$&$4.6^\circ$\\
				\midrule
				\bottomrule
			\end{tabular}
			\caption{\textbf{Quantitative performance of \NIPH{} using different sampling parameters}}
			\label{tab:parametersMSE}
		\end{center}
	\end{figure}
	\begin{figure}[tb!]
		\begin{center}
			\begin{subfigure}{0.50\columnwidth}
				\includegraphics[width=\columnwidth]{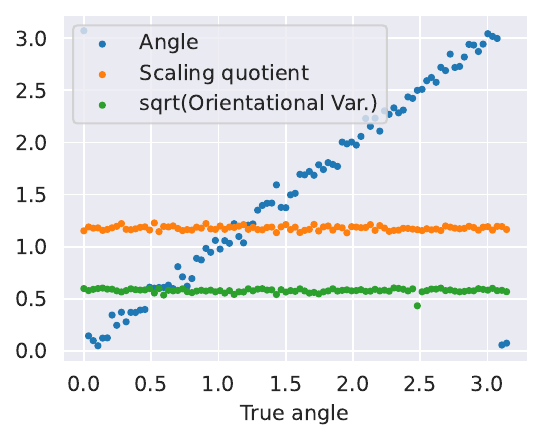}
			\end{subfigure}
			\caption{\textbf{Performance of \NIPH{} for a combination of the probing scaling factors %$\{1.5,2,3,4,8\}$
					and $8$ probing directions.} We have a ground truth orientational standard deviation of $0.5$, scaling factor of $1.5$ and varying angle.
			}
			\label{fig:Combined Scales}
		\end{center}
	\end{figure}
	In this section, we will analyse how different parameter choices affect the performance of \NIPH{}.
	The presented form of \NIPH{} does not have many parameters, and mainly the choice of sampling directions and probing scaling factor has an impact on the result.
	In \Cref{fig:DifferentProbingScales}, we have repeated the experiment of \Cref{fig:SynthExpVaryingPhi} with $8$ probing angles and varying probing scaling factors.
	We make the following observations:
	\begin{enumerate}[i)]
		\item Picking a probing scaling factor that is not above the true scaling factor of the data set results in \NIPH{} failing to extract meaningful scaling information. However, the extraction of the orientational variance and the underlying orientation still performs well.
		\item Picking probing scaling factors that are significantly larger than the underlying scaling factor results in bad overall performance.
		\item There is a wide range of possible probing scaling directions for which \NIPH{} produces good results.
		\item Sampling with different probing scaling factors during one iteration of \NIPH{} produces good results comparable with the best individual parameter choice, see \Cref{fig:Combined Scales}. This solves the problem of choosing a fixed probing scaling factor in practice.
	\end{enumerate}
	In \Cref{fig:AngleCounts}, we have repeated the experiment of \Cref{fig:SynthExpVaryingPhi} with a varying number of probing angles and a fixed probing scaling direction of $2$.
	The quality of the approximation increases with the number of probing directions used.
	In our case, the the quality increase above $8$ scaling directions is only minuscule.
	In the case of a low number of sampling directions, the errors are periodic.
	We compare the root mean squared errors of the different parameter selections in \Cref{tab:parametersMSE}.
	
	\newpage
	\section{Figures}
	\FloatBarrier
	\begin{figure}[htb!]
		\begin{center}
			\begin{subfigure}{0.95\columnwidth}
				\includegraphics[width=\columnwidth]{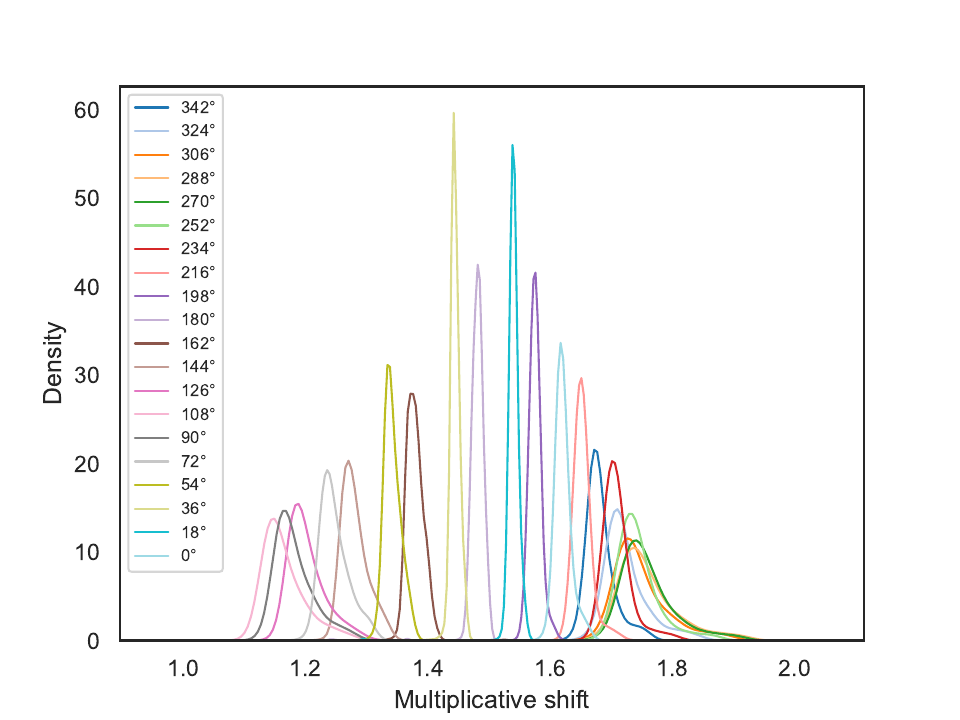}
			\end{subfigure}
			\begin{subfigure}{0.95\columnwidth}
				\includegraphics[width=\columnwidth]{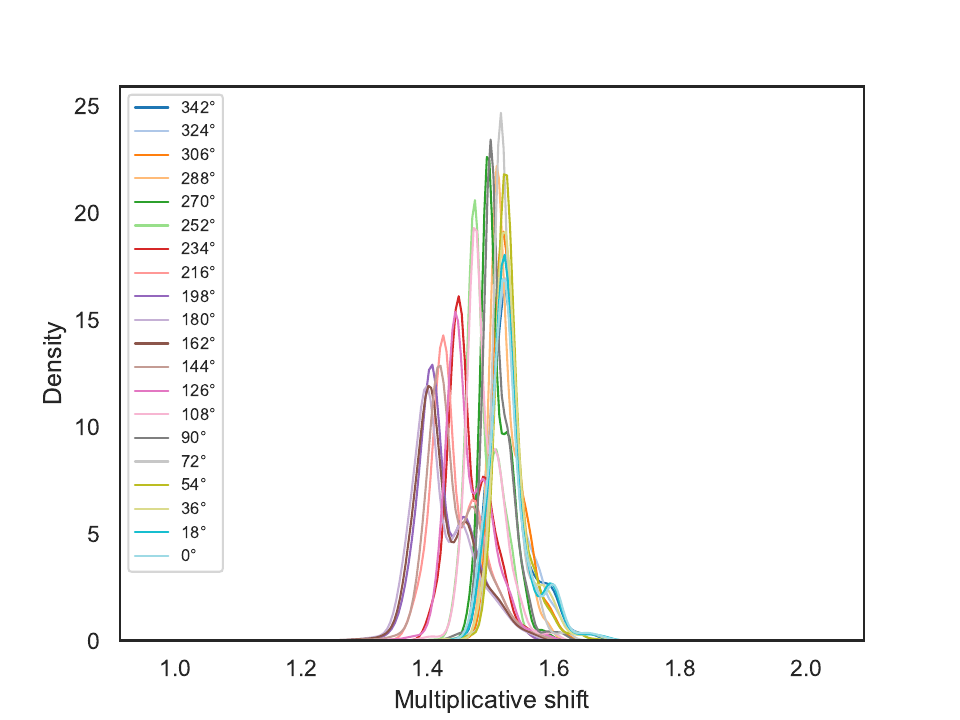}
			\end{subfigure}
			\caption{\textbf{Mult. Shift Diagrams used to construct \Cref{fig:RoadNetworks}}
				\emph{Top:} Residential Roads in Manhattan. \emph{Bottom:} Residential Roads in Tübingen.
			}
			\label{fig:RoadNetworksMS}
		\end{center}
	\end{figure}
%

%\begin{figure}[tb!]
%\begin{center}
%\includegraphics[width=0.7\columnwidth]{figs/Manhattan_road_network_big.png}
%\caption{\textbf{Larger version of \Cref{fig:RoadNetworks}}
%	Although there are some primary roads in a west-easterly direction, the majority of the primary roads is oriented in a north-southernly direction. Notice that openstreetsmaps counts some of the avenues, e.g.\ Broadway and Park Avenue, as two separate lanes, increasing their importance for the calculations. Note that we didn't include the motorways in the previous smaller figure to improve legibility.}
%\label{fig:NYCBig}
%\end{center}
%\end{figure}
%\begin{figure}[tb!]
%\includegraphics[width=\columnwidth]{figs/Tubingen_road_network.png}
%\caption{\textbf{Larger version of \Cref{fig:RoadNetworks}.}}
%\label{fig:TubBig}
%\end{figure}
\end{document}